\documentclass[11pt]{amsart}
\usepackage{amssymb,mathrsfs,graphicx,enumerate}
\usepackage{amsmath,amsfonts,amssymb,amscd,amsthm,bbm}

\usepackage[english]{babel}
\usepackage{subeqnarray}
\usepackage{cases}
\usepackage{pst-all}
\usepackage{pstricks}
\usepackage{graphicx}
\usepackage{float}
\usepackage[labelfont=bf,labelsep=space]{caption}
\usepackage{color}
\usepackage{tikz}
\usepackage[title]{appendix}

\usepackage{graphicx,colortbl}

\title[The Cucker-Smale model with a time-delay on a general digraph]{Interplay of time-delay and velocity alignment in the Cucker-Smale model on a general digraph}

\author[Dong]{Jiu-Gang Dong}
\address[Jiu-Gang Dong]{\newline Department of Mathematics \newline Harbin Institute of Technology, China}
\email{jgdong@hit.edu.cn}

\author[Seung-Yeal Ha]{Seung-Yeal Ha}
\address[Seung-Yeal Ha]{\newline Department of Mathematical Sciences and Research Institute of Mathematics, Seoul National University, Seoul 151-747, Korea}
\email{syha@snu.ac.kr}

\author[Kim]{Doheon Kim}
\address[Doheon Kim]{\newline Department of Mathematical Sciences, \newline Seoul National University, Seoul 08826, Korea (Republic of)}
\email{dohun930728@snu.ac.kr}

\newtheorem{theorem}{Theorem}[section]
\newtheorem{lemma}{Lemma}[section]
\newtheorem{corollary}{Corollary}[section]
\newtheorem{proposition}{Proposition}[section]
\newtheorem{remark}{Remark}[section]

\newtheorem{definition}{Definition}[section]

\newcommand{\bbr}{\mathbb R}

\newcommand{\bx}{\mbox{\boldmath $x$}}

\newcommand{\bv}{\mbox{\boldmath $v$}}

\newcommand{\R}{\mathbb R}
\newcommand{\N}{\mathcal N}
\newcommand{\G}{\mathcal G}
\newcommand{\V}{\mathcal V}
\newcommand{\C}{\mathcal C}
\newcommand{\D}{\mathcal D}

\newcommand{\E}{\mathcal E}

\newcommand{\NN}{\mathbb N}

\def\bx{\boldsymbol x}

\def\bv{\boldsymbol v}

\def\dist{\mathsf{dist}}

\begin{document}

\date{\today}


\thanks{\textbf{Acknowledgment.} The work of S.-Y. Ha is supported by
the Samsung Science and Technology Foundation under Project Number SSTF-BA1401-03}

\begin{abstract}
We study dynamic interplay between time-delay and velocity alignment in the ensemble of Cucker-Smale (C-S) particles(or agents) on time-varying networks which are modeled by digraphs containing spanning trees. Time-delayed dynamical systems often appear in mathematical models from biology and control theory, and they have been extensively investigated in literature. In this paper, we provide sufficient frameworks for the mono-cluster flocking to the continuous and discrete C-S models, which are formulated in terms of system parameters and initial data. In our proposed frameworks, we show that the continuous and discrete C-S models exhibit exponential flocking estimates. For the explicit C-S communication weights which decay algebraically, our results exhibit threshold phenomena depending on the decay rate and depth of digraph. We also provide several numerical examples and compare them with our analytical results.
\end{abstract}
\maketitle \centerline{\date}

\tableofcontents

\section{Introduction}
\setcounter{equation}{0}
The terminology ``{\it flocking}" represents a collective phenomenon in which the ensemble of flocking particles organize into an ordered motion using only limited information and simple rules. This flocking behavior appears in many biological complex systems, e.g., flocking of birds, schooling of fish and swarming of bacteria, etc. The flocking problem has received lots of attention from researchers in several scientific disciplines such as computer science, biology and statistical physics etc in the last decades \cite{BCGLOPVZ08, CKFS05, CS07,CS071, JLM03, OS06, Rey87, TJJ07, VKJS95, VZ12} due to its engineering applications in mobile sensor networks, cooperative robotics and formation flying spacecrafts~\cite{BLH01,CMKB04,PEG09}. After Vicsek's seminal work \cite{VKJS95}, several mathematical models have been proposed in literature \cite{VZ12}. Among them, our main interest in this paper lies on the Newton type particle system proposed by Cucker and Smale \cite{CS07} in a decade ago.  In this model, the interaction weight between particles is assumed to be an algebraically decaying function depending on the Euclidean distance between particles. One of the novel features in the Cucker-Smale(C-S) model is that emergence of flocking exhibits a threshold phenomenon, for example, when the decay rate is slower than that of Coloumb potential, the emergence of global flocking is guaranteed for any initial configuration, whereas if the decay rate is faster than that of Coloumb potential, the emergence of global flocking holds only for some restricted class of initial configuration. As discussed in the abstract, our main focus in this paper is the interplay between communication time-delay and velocity alignment in the C-S ensemble.

Next, we briefly discuss our continuous-time C-S type model (see Section \ref{sec:2.1} for details). Let $(\bx_i(t),\bv_i(t)) \in \bbr^{2d}$ be the position and velocity of the $i$-th C-S particle, respectively, and for $\tau > 0$, the time-delay function $\tau_{ij}: \bbr_+ \to [0, \tau]$ is bounded and continuous function representing the communication time-delay in the information flow from the $j$-th particle to the $i$-th particle, and we also denote the network topology via the connection matrix $(\chi_{ij})$:
\[ \chi_{ij} = \begin{cases}
1, \quad & \mbox{if $j$ transmits information to $i$}, \cr
0, \quad & \mbox{otherwise,}
\end{cases}
\]
and the communication weights between $i$-th and $j$-th particles are registered by the communication weight function $\psi(\|\bx_j(t -\tau_{ij}(t)) - \bx_i(t) \|)$ as for the general C-S model \cite{H-Liu}. Under the aforementioned setup, the dynamics of $(\bx_i(t),\bv_i(t))$ is governed by the following Cauchy problem:
\begin{equation} \label{A-1}
\begin{cases}
\displaystyle \dot{\bx}_i(t)=\bv_i(t), \quad t > 0,~~i = 1, \cdots, N, \\
\displaystyle \dot{\bv}_i(t) = \sum\limits_{j = 1}^{N}  \chi_{ij} \psi(\|\bx_j(t -\tau_{ij}(t)) - \bx_i(t) \|) (\bv_j(t-\tau_{ij}(t))-\bv_i(t)), \\
(\bx_i, \bv_i)(t) = (\bx_i^0(t), \bv_i^0(t)), \quad -\tau \leq t \leq 0,
\end{cases}
\end{equation}
where the weight function $\psi$ is assumed to be bounded, locally Lipschitz continuous positive non-increasing functions defined on the nonnegative real numbers:
\begin{equation} \label{A-2}
0 < \psi(r) \leq \kappa , \quad r \geq 0, \quad (\psi(r_1)  - \psi(r_2)) (r_1 - r_2) \leq 0, \quad r_1, r_2 \geq 0.
\end{equation}
and initial data $(\bx_i^0(t), \bv_i^0(t))$ are assumed to be bounded and continuous in the time-interval $[-\tau, 0]$.

The well-posedness of the time-delayed system \eqref{A-1} and \eqref{A-2} can be found in \cite{Hale-L}, which proves the existence and uniqueness of the classical solutions to system \eqref{A-1} for any continuous initial conditions. For the complete network, zero time-delay and C-S communication weight:
\[ \chi_{ij} = 1, \quad \tau_{ij} \equiv 0, \quad  1 \leq i, j \leq N, \qquad \psi(r) = \frac{1}{(1 + r^2)^{\beta}}, \quad \beta \geq 0, \]
system \eqref{A-1} was first proposed in \cite{CS07,CS071} and the threshold phenomena between global and local flocking have been observed. Since then, these flocking estimates have been generalized for a general non-increasing communication weight in \cite{H-Liu} which corresponds to the zero-time delayed case \eqref{A-1} with $\tau_{ij}  =0,~\chi_{ij} =1$.  Emergent dynamics  and related topics have been extensively studied from many perspectives, e.g., collision avoidance \cite{A-C-H-L, CD10}, random effects \cite{A-H, C-M, HLL09}, mean-field limit and measure-valued solutions \cite{B-C-C, C-C-R}, collective dynamics \cite{C-F-R-T, CS07, CS071, F-H-T,  H-Liu, H-T, She}, local flocking \cite{C-H-H-J-K1, H-K-Z}, bonding force \cite{P-K-H}, generalized flocking \cite{M-T}, singular and hyperbolic limits \cite{P-S}, kinetic equation \cite{D-F-T}, application to flight formation \cite{PEG09} and flocking with leaders \cite{L-X} etc.

Time delays often appear in practical applications due to position and velocity data processing and transmission, as pointed out by the recent works~\cite{VVSTSNV14, VVTSSNV14}, which have performed some real experiments to achieve several fundamental collective flight tasks with autonomous aerial robots.  The C-S model \eqref{A-1} - \eqref{A-2}  with time-delay has been studied in literature. For example, in \cite{CH17,LW14,PT17}, the C-S model with a single constant time delay and time varying delays under the complete graph (i.e., all-to-all interactions) is considered, respectively.   In  \cite{PV17},  a flocking condition is derived for the C-S model with a single constant delay under hierarchical leadership (a special class of digraphs).
 However, to the best of our knowledge, the effect of multiple time varying delays on the C-S model under general digraphs has not been addressed in literature. This is our main motivation to introduce a delayed C-S model under general digraphs where communication delays are allowed to be time-varying but bounded by $\tau$.  Before we discuss our main results, we introduce several notation:
\begin{align}
\begin{aligned} \label{A-3}
&  \bx_i :=(x^1_{i},\cdots,x^d_{i}) \quad \mbox{and} \quad \bv_i :=(v^1_{i},\cdots, v^d_{i}), \quad i = 1, \cdots, N, \cr
& \bar{v}^k(t) :=\max_{\substack{1\leq i\leq N\\t-\tau\leq s\leq t}}v^k_{i}(s), \quad
\underline{v}^k (t) :=\min_{\substack{1\leq i\leq N\\t-\tau\leq s\leq t}}v^k_{i}(s), \quad k = 1, \cdots, d, \cr
& \D^k(t) := \bar{v}^k(t) - \underline{v}^k (t), \quad \D(t)=\max_{1\leq k \leq d}\D^k(t), \quad X(0) :=\max_{\substack{\chi_{ij} > 0 \\-\tau\leq s\leq0}}\|\bx_i(0)-\bx_j(s)\|.
\end{aligned}
\end{align}
Then, our sufficient condition for the mono-cluster flocking can be summarized in the following simple relation. For given initial data, time-delay and network topology, if there exists a positive constant $\rho$ such that
\[  \D(0) \leq  \frac{e^{-n^{\infty} \kappa {\gamma_g}(3\tau+2)}}{ 2\sqrt{d}\gamma_g(2\tau+1)(1+ n^{\infty} \kappa)^{\gamma_g}} \rho (\psi(X(0)+\rho))^{\gamma_g}, \]
where $\gamma_g$ and $n^{\infty}$ are the depth of the spanning tree and maximal size of neighbor set defined in \eqref{B-0} and \eqref{C-0}, respectively, then the relative velocities and positions satisfy mono-cluster flocking estimates (see Theorem \ref{T3.1} for continuous model):
\[ \sup_{t\geq0}\|\bx_i(t)-\bx_j(t)\|<\infty, \qquad   \lim_{t\rightarrow\infty}\|\bv_i(t)-\bv_j(t)\|=0, \textrm{ for all }1\leq i,j\leq N.  \]
For the discrete C-S model, similar result can be obtained.
\newline

The rest of this paper is organized as follows.  In Section \ref{sec:2}, we present preliminaries on graph theory and discuss our C-S model with time-delay on a digraph. In Section \ref{sec:3}, we present our two main results on flocking estimates for the continuous and discrete-time C-S models. In Section~\ref{sec:4}, we provide a proof of our first main result on the mono-cluster flocking dynamics of the continuous-time C-S model. In Section \ref{sec:5}, we provide a discrete analogue of the continuous time result. In Section \ref{sec:6}, we provide several numerical examples and compare them with our analytical results. Finally, Section \ref{sec:7} is devoted to the brief summary of our main results and future directions.

 \section{Preliminaries} \label{sec:2}
\setcounter{equation}{0}
In this section, we briefly describe our backbones for the interplay between a dynamic network and the flocking system on it. To fix the idea, we model network dynamics and system dynamics by a directed graph (in short, digraph) with a spanning tree, time-varying capacities, and the C-S model with time-delay, respectively.
\subsection{Network dynamics}\label{sec:2.1}
We discuss a theoretical minimum in a digraph which models the neighbor set of each agent and communication weights between agents. We will identify $i$ as the $i$-th agent.  A weighted digraph $\G=(\V(\G),\E(\G), \C(\G) )$ consists of a finite set $\V(\G)=\{1,\ldots, N\}$ of {\em vertices (nodes)}, a set $\E(\G) \subset\V(\G) \times \V(\G)$ of {\em arcs}, and a matrix $ \C(\G)$ whose elements represent communication (interaction) weights between agents.
If $(j,i)\in\E(\G)$, we say that $j$ is a {\em neighbor} of $i$, and we denote the neighbor set of the vertex $i$ by
$\N_i:=\{j:(j,i)\in\E(\G)\}$. We set
\[ \chi_{ij} = \begin{cases}
1, \quad & \mbox{if $j$ is a neighbor of $i$}, \cr
0, \quad & \mbox{otherwise.}
\end{cases}
\]

Throughout the paper, we exclude a self loop, i.e., $i \not \in \N_i$ for all $1\leq i \leq N$.
A {\em path} in $\G$ from $i_0$ to $i_p$ is a
sequence $i_0, i_1, \ldots,i_p$ of distinct vertices such that each successive pair of
vertices is an arc of the digraph. The integer $p$ (the number of its
arcs) is the {\em length} of the path.  If there exists a path from $i$ to $j$, then vertex~$j$ is said to be
{\em reachable} from vertex~$i$, and we define the distance from $i$ to $j$, $\dist(i,j)$,
as the length of a shortest path from $i$ to $j$.
A digraph~$\G$ is said to be
{\em strongly connected} if each vertex can be reachable from any
other vertex, whereas a weaker concept is that of the existence of a spanning tree.
We say that $\G$ has a {\em spanning tree} if we can find a vertex
(called a {\em root}) such that any
other vertex of~$\G$ is reachable from it.  For each root $r$ of digraph $\G$ with a spanning tree, we define $\max_{j\in\V}\dist(r,j)$
as the depth of the spanning tree of $\G$ rooted at $r$. The {\em smallest depth} of $\G$ is
\begin{equation} \label{B-0}
 \gamma_g:=\min_{r~ \mbox{is a root}}\max_{j\in\V}\dist(r,j).
\end{equation}
Thus, any agent in a digraph with a spanning tree can be connected to one of a root agent via no more than $\gamma_g$ intermediate agents. It is trivial to see $\gamma_g \leq N-1$. Moreover, $\gamma_g \leq\lfloor N/2 \rfloor$ when $\G$ is  undirected~\cite{AFH09}. We understand that $\gamma_g=\infty$, when $\G$ does not have a spanning tree. The $(j, i)$-element of the capacity matrix $\C(\G) = (c_{ij})$ denotes the interaction weight (amount of information) from the $j$-vertex  to $i$-vertex. In this paper, we use
\begin{equation} \label{B-0-1}
  c_{ij} = \psi(\|\bx_j(t-\tau_{ij}(t)) - \bx_i(t) \|)
 \end{equation}
 as an interaction weight.
\subsection{System dynamics} \label{sec:2.2}
Consider an autonomous multi-agent system consisting of $N$ interacting agents moving in Euclidean space $\R^d$, where the agents are not necessarily equal in social status and a leadership hierarchy can exist between agents.

Let $\N_i$ be the neighbor set of the agent $i$, consisting of the agents that directly influence $i$, and let $\tau_{ij}$ be a nonnegative bounded function representing the communication time-delay from $j$ to $i$ at time $t$. The communication network of the flock is denoted by a digraph $\G=(\V,\E, \C)$ such that $(j,i)\in\E$ if and only if $j\in\N_i$. We see that communication delays $\tau_{ij}(\cdot)$ appear in both position and velocity transmissions. Here self-delay is not allowed:
\begin{equation} \label{B-0-2}
 \tau_{ii}(\cdot)\equiv0 \quad \mbox{ for all}~ i \in \{1, \cdots, N\}.
\end{equation}
The weight function $\psi$ satisfies the conditions \eqref{A-2}. Since each agent adjusts its velocity by the delayed relative state measurements, the differences between its current state and that of other agents undergo some communication delays. Under the setting \eqref{B-0-1} and \eqref{B-0-2}, as a system dynamics, we consider the C-S model with time-dependent time-delay $\tau_{ij} = \tau_{ij}(t)$:
\begin{align*}
\begin{aligned} \label{B-1}
\dot{\bx}_i(t) &=\bv_i(t), \quad t > 0, \quad i = 1, \cdots, N, \\
\dot{\bv}_i(t) &=\sum\limits_{j=1}^{N} \chi_{ij} \psi(\|\bx_j(t -\tau_{ij}(t)) - \bx_i(t) \|) (\bv_j(t-\tau_{ij}(t))-\bv_i(t)).
\end{aligned}
\end{align*}
Note that the original C-S model in \cite{CS07} corresponds to the following setting:
\[ \chi_{ij} = 1, \quad 1 \leq i, j \leq N, \qquad \psi(r) = \frac{\kappa}{N(1 + r^2)^{\beta}}. \]
Throughout the paper, we assume that the time-varying delays $\tau_{ij}(\cdot)$ are uniformly bounded, i.e., there exists $\tau\geq0$ such that
\[ 0\leq\tau_{ij}(t)\leq\tau \quad \mbox{for all $t>0$ and $i, j$}. \]
Next, we consider the corresponding discrete-time analogue of the continuous system \eqref{A-1} which can be obtained via the forward Euler scheme:
\begin{align*}\label{B-2}
\begin{aligned}
\bx_i[t+1]&=\bx_i[t]+h\bv_i[t], \quad t \in \NN\cup\{0\}, \\
\bv_i[t+1]&=\bv_i[t]+ h\sum_{j = 1}^{N} \chi_{ij} \psi(\|\bx_j[t -\tau_{ij}[t]] - \bx_i[t] \|)
(\bv_j[t-\tau_{ij}[t]]-\bv_i[t]),
\end{aligned}
\end{align*}
where $h= \Delta t$ is the time step, $\bx_i[t] :=\bx_i(ht)$ and $\bv_i[t] :=\bv_i(ht)$ with $t\in\NN\cup\{0\}$. We also assume that time-delays $\tau_{ij}$ are integer-valued functions.

\section{Description of main results}
\setcounter{equation}{0} \label{sec:3}
In this section, we present our main results for the continuous and discontinuous C-S models, respectively. First, we recall the concept of flocking in the following definition.
\begin{definition} \label{D3.1}
Let  ${\mathcal Z}(t) := \{(\bx_i(t),\bv_i(t))\}_{i=1}^N$ be time-dependent C-S configuration. Then, ${\mathcal Z}(t)$ exhibits asymptotic {\it mono-cluster flocking} if the following two conditions hold:
\begin{enumerate}
\item
The relative velocities tend to zero asymptotically:
\[ \lim_{t\rightarrow\infty}\|\bv_i(t)-\bv_j(t)\|=0, \textrm{ for all }1\leq i,j\leq N. \]
\item
The relative positions are uniformly bounded in time:
\[
\sup_{t\geq0}\|\bx_i(t)-\bx_j(t)\|<\infty \textrm{ for all }1\leq i,j\leq N. \]
\end{enumerate}
\end{definition}
In the following two subsections, we present two main results for the emergence of mono-cluster flockings to the continuous and discrete C-S models. The detailed proofs will be provided in next two sections.
\subsection{The continuous C-S model}  \label{sec:3.1} Consider the following Cauchy problem:
\begin{equation} \label{C-1}
\begin{cases}
\dot{\bx}_i(t) =\bv_i(t),\quad t > 0, \quad i = 1, \cdots, N, \\
\dot{\bv}_i(t) =\sum\limits_{j = 1}^{N} \chi_{ij} \psi(\|\bx_j(t -\tau_{ij}(t)) - \bx_i(t) \|)
(\bv_j(t-\tau_{ij}(t))-\bv_i(t))\\
(\bx_i, \bv_i)(t) = (\bx^0_i(t), \bv^0_i(t)),\quad -\tau \leq t \leq 0.
\end{cases}
\end{equation}
Note that the neighbor set ${\mathcal N}_i$ of $i$-th particle satisfies
\[ {\mathcal N}_i := \{ j~: j \not = i, \quad \chi_{ij} > 0 \}. \]
We put the following standing assumption on the network structure throughout the paper: \newline

\noindent {\bf Assumption}: Suppose that $\G$ has a spanning tree and we denote its smallest depth by $\gamma_g$. Let $n^{\infty}$ be the maximal cardinality of the neighbor sets $\N_i$ for all $1 \leq i\leq N$:
\begin{equation} \label{C-0}
  n^{\infty} := \max_{1\leq i \leq N} |{\mathcal N}_i|,
\end{equation}
where $|A|$ is the cardinality of the set $A$.  Before we state our first main result, we set
\[ C_\infty :=  \frac{e^{-n^{\infty} \kappa {\gamma_g}(3\tau+2)}}{ 2\sqrt{d}\gamma_g(2\tau+1)(1+ n^{\infty} \kappa)^{\gamma_g}}. \]
Our first main result can be stated as follows.
\begin{theorem}\label{T3.1}
Suppose that for a given initial data $\{(\bx_i^0, \bv_i^0) \}$, maximal size of time-delay $\tau$ and network topology, there exists a positive constant $\rho>0$ such that
\begin{equation}\label{C-1-1}
 \D(0)\leq  C_\infty \rho (\psi(X(0)+\rho))^{\gamma_g}.
\end{equation}
Then,  for any solution $\{ (\bx_i, \bv_i) \}_{i=1}^{N}$ to \eqref{C-1}, we have an asymptotic mono-cluster flocking estimate:
\[ \sup_{t\geq0}\|\bx_i(t)-\bx_j(t)\|<\infty, \quad \lim_{t\rightarrow\infty}\|\bv_i(t)-\bv_j(t)\|=0, \textrm{ for all }~1\leq i,j\leq N.  \]
\end{theorem}
\begin{proof}
Since the proof is rather lengthy, we postpone it to Section \ref{sec:4.1}.
\end{proof}
\begin{remark} \label{R3.1}
1. In the literature, the study on the C-S type models with time-delay is all restricted to the special network topologies.
 In \cite{CH17, C-L, E-H-S, PT17}, flocking estimates for the C-S type models with time-delay on a complete network have been discussed using Lyapunov type approach. In a recent work \cite{PV17}, the C-S type model with time-delay in velocity component has been also discussed on a hierarchical network where the asymptotic velocities of particles are given by that of leader. Thus, our current framework is general enough to include complete network and hierarchical network.  Compared to the global approach based on the Lyapunov functionals in aforementioned literature, our current approach in Section \ref{sec:4} is local in the sense that we can see explicitly how information flow can affect the velocity flocking.

\noindent 2. Note that the result stated in Theorem \ref{T3.1} only says that the relative velocities tend to zero asymptotically, and our methodology employed in this paper does not provide an asymptotic flocking velocities.
\end{remark}
As a direct application of Theorem \ref{T3.1}, we have the following result for the C-S communication weight :
\[ \psi_{cs}(r) = \frac{\kappa}{(1 + r^2)^{\beta}}. \]
\begin{corollary}\label{C3.1}
Suppose that initial data $\{(\bx_i^0, \bv_i^0) \}$  and parameters $\gamma_g,~\beta$ satisfy one of the following conditions:
\[
\begin{cases}
\displaystyle  & 2 \beta \gamma_g < 1, \cr
\displaystyle   & 2 \beta \gamma_g  =  1,\quad \D(0) < \frac{e^{-n^{\infty} \kappa \gamma_g (3\tau+2)} \kappa^{\gamma_g}}{2\sqrt{d}\gamma_g (2\tau+1)(1+n^{\infty} \kappa)^{\gamma_g}}, \cr
\displaystyle  & 2 \beta \gamma_g  > 1, \quad \D(0) \leq R(\rho_+) := \frac{C_\infty \rho_+ \kappa^{\gamma_g}}{[ 1 + (X(0) + \rho_+)^2 ]^{\beta \gamma_g}},
 \end{cases}
 \]
where $\rho_+$ is given by the following relations:
\[
\rho_+ := \frac{ X(0) (1-\beta \gamma_g) + \sqrt{  \beta^2 \gamma_g^2 X(0)^2 + 2 \beta \gamma_g -1}}{2\beta \gamma_g-1}.
\]
Then, for any solution $\{ (\bx_i, \bv_i) \}_{i=1}^{N}$ to \eqref{C-1} with $\psi = \psi_{cs}$, we have an asymptotic mono-cluster flocking estimate:
\[ \sup_{t\geq0}\|\bx_i(t)-\bx_j(t)\|<\infty, \quad \lim_{t\rightarrow\infty}\|\bv_i(t)-\bv_j(t)\|=0, \textrm{ for all }~1\leq i,j\leq N.  \]
\end{corollary}
\begin{proof}
We postpone its proof in Section \ref{sec:4.3}.
\end{proof}
\begin{remark} \label{R3.2}
As in the original C-S model, our results also show the existence of the critical exponent for the parameter $\beta$,  below which unconditional convergence holds. It is given in terms of the connectivity of the communication network.  When we consider the all-to-all network, we see that the critical exponent is $1/2$, the same with that in \cite{CH17, CS07}. In \cite{PV17},  for the C-S model with a single constant delay under hierarchical leadership,  unconditional convergence to flocking is established for $\beta\leq 1/2$.
\end{remark}

\subsection{The discrete C-S model} \label{sec:3.2} In this subsection, we consider the discrete-time C-S model:
\begin{align}
\begin{aligned} \label{C-2}
\bx_i[t+1]&=\bx_i[t]+h\bv_i[t], \quad t \in \NN\cup\{0\}, \\
\bv_i[t+1]&=\bv_i[t]+ h\sum_{j = 1}^{N} \chi_{ij} \psi(\|\bx_j[t -\tau_{ij}[t]] - \bx_i[t] \|)
(\bv_j[t-\tau_{ij}[t]]-\bv_i[t]).
\end{aligned}
\end{align}
Here $\tau_{ij}[t],\tau\in\NN\cup\{0\}$ and $0\leq\tau_{ij}[t]\leq \tau$ for all $t\in\NN\cup\{0\}$. We set

\[ {\bar C}_\infty :=  \frac{(1-h n^{\infty} \kappa)^{\gamma_g (3\tau +1)} h^{\gamma_g}}{ 2\sqrt{d} h \gamma_g(2\tau+1)(1+ n^{\infty} \kappa)^{\gamma_g}}. \]

\begin{theorem}\label{T3.2}
Suppose that initial data $\{(\bx_i^0, \bv_i^0) \}$ and parameters $h,~\kappa,~\gamma_g,~\beta$ satisfy the following conditions: there exists a positive constant $\rho>0$ such that
\[ 0<\kappa h<\frac{1}{n^{\infty}}, \quad  \D[0]\leq  {\bar C}_\infty \rho (\psi(X[0]+\rho))^{\gamma_g}. \]
Then,  for any solution $\{ (\bx_i, \bv_i) \}_{i=1}^{N}$ to \eqref{C-2}, we have an asymptotic mono-cluster flocking estimate:
\[ \sup_{t \in \NN}\|\bx_i[t]-\bx_j[t]\|<\infty, \quad \lim_{t\rightarrow\infty}\|\bv_i[t]-\bv_j[t]\|=0, \textrm{ for all }~1\leq i,j\leq N.  \]
\end{theorem}
\begin{proof}
Since the proof is rather lengthy, we postpone it to Section \ref{sec:5.1}.
\end{proof}
As a direct application of Theorem \ref{T3.2}, we have the following flocking estimate for the C-S communication weight.
\begin{corollary}\label{C3.2}
Suppose that initial data $\{(\bx_i^0, \bv_i^0) \}$  and parameters $h,~\kappa,~\gamma_g,~\beta$ satisfy the following conditions:
\[ 0<\kappa h<\frac{1}{n^{\infty}}, \]
and
\[
\begin{cases}
\displaystyle  & 2 \beta \gamma_g < 1, \cr
\displaystyle   & 2 \beta \gamma_g  =  1, \quad \D[0] < \frac{(1-hn^{\infty} \kappa)^{\gamma_g(3\tau+1)}h^{\gamma_g-1} \kappa^{\gamma_g}}{2\sqrt{d}\gamma_g(2\tau+1)(1+n^{\infty} \kappa)^{\gamma_g}},  \cr
\displaystyle & 2 \beta \gamma_g  > 1, \quad \D[0] \leq \frac{ (1-h n^{\infty} \kappa)^{\gamma_g(2\tau+1)}h^{\gamma_g-1} \rho_+}{ 2\sqrt{d}\gamma_g(2\tau+1)(1+n^{\infty} \kappa)^{\gamma_g} } (\psi( X[0] + \rho_+   ))^{\gamma_g},
\end{cases}
 \]
where $\rho_+$ is given by the following relation:
\[
 \rho_+ := \frac{ X[0] (1-\beta \gamma_g) + \sqrt{  \beta^2 \gamma_g^2 X[0]^2 + 2 \beta \gamma_g -1}}{2\beta \gamma_g - 1}.
\]
Then, for any solution $\{ (\bx_i[t], \bv_i[t]) \}_{i=1}^{N}$ to \eqref{C-1} with $\psi = \psi_{cs}$, we have an asymptotic mono-cluster flocking estimate:
\[ \sup_{t\geq0}\|\bx_i[t]-\bx_j[t]\|<\infty, \quad \lim_{t\rightarrow\infty}\|\bv_i[t]-\bv_j[t]\|=0, \textrm{ for all }~1\leq i,j\leq N.  \]

\end{corollary}
\begin{proof}
If we replace $C_\infty$ by ${\bar C}_\infty$, then the proof is exactly the same as that of Corollary \ref{C3.1}.
\end{proof}
\begin{remark}
For the case without delays, i.e., $\tau=0$, Theorems~\ref{T3.1} and \ref{T3.2} reduce to the results obtained in \cite{DQ17} but by a different approach. Also, the exponential convergence rate is explicitly given in our proof, showing that the delay can degrade the rate of convergence. We verify the theoretical analysis by numerical simulations in the next section.
\end{remark}

\section{The continuous Cucker-Smale model}\label{sec:4}
\setcounter{equation}{0}
In this section, we present an asymptotic flocking estimate for the continuous-time delayed C-S model~\eqref{C-1}.

\subsection{Basic estimates} \label{sec:4.1} In this subsection, we present several lemmas to be used in the proof of Theorem \ref{T3.1}. For the flocking estimate, it is convenient to deal with the system~\eqref{C-1} componentwise. Let $(x_i^k, v_i^k)$ be a $k$-th component of the position-velocity coordinate of the $i$-th agent. Then, it satisfies
\begin{align}
\begin{aligned} \label{D-1}
\dot{x}^k_{i}(t)&=v^k_{i}(t), \quad t > 0, \quad i = 1, \cdots, N,~~k = 1, \cdots, d, \\
\dot{v}^k_{i}(t)&=\sum_{j=1}^{N}  \chi_{ij} \psi(\|\bx_j(t -\tau_{ij}(t)) - \bx_i(t) \|)
(v^k_{j}(t-\tau_{ij}(t))-v^k_{i}(t)).
\end{aligned}
\end{align}
In next lemma, we study the monotonicity of $\bar{v}^k$ and $\underline{v}^k$.
\begin{lemma}\label{L4.1}
Let  $\{(\bx_i(t),\bv_i(t))\}_{i=1}^N$ be a solution to system~\eqref{D-1}. Then, $\bar{v}^k(t)$ and $\underline{v}^k(t)$ defined in \eqref{A-3} are nonincreasing and nondecreasing, respectively.
\end{lemma}
\begin{proof}
(i) First, we show that $\bar{v}^k(t)$ is non-increasing. Suppose that this is not the case. Then there must be some instant $t^*$ such that the upper right-hand Dini derivative of $\bar{v}^k(t)$ at $t^*$ is positive, i.e., $D^+\bar{v}^k(t^*)>0$ which means that $\bar{v}^k(t)$ is strictly increasing at $t^*$. Thus, by definition of $\bar{v}^k(t)$, there exists an agent $i^*$ such that $\bar{v}^k(t^*)=v^k_{i^*}(t^*)$ since otherwise  $\bar{v}^k(t)$ is not increasing at $t^*$ by continuity. However, system~$\eqref{D-1}_2$ implies
\[
\dot{v}^k_{i^*}(t^*) = \sum_{j=1}^{N}  \chi_{i^*j} \psi(\|\bx_j(t^* -\tau_{i^*j}(t^*)) - \bx_{i^*}(t^*) \|)  (v^k_{j}(t^*-\tau_{i^*j}(t^*))-v^k_{i^*}(t^*)) \leq 0,
\]
which yields contradiction. Thus, we have the monotone decreasing property of $\bar{v}^k(t)$. \newline

\noindent (ii) The case for $\underline{v}^k(t)$ can be treated similarly. Hence, we omit its details.
\end{proof}
\begin{remark} \label{R4.1} The result of Lemma \ref{L4.1} yields the monotone decreasing property of ${\mathcal D}^k(t)$:
\[ ({\mathcal D}^k(t_2) - {\mathcal D}^k(t_1)) (t_2 - t_1) \leq 0, \quad t_1, t_2 \in (0, \infty). \]
\end{remark}

\bigskip

In next lemma, we show that there exists a strictly increasing upper envelope profile for $v_i^k$ which may oscillate due to the time-delay.
\begin{lemma}\label{L4.2}
Let  $\{(\bx_i(t),\bv_i(t))\}_{i=1}^N$ be a solution to system~\eqref{D-1}.  If for $t_2\geq t_1\geq0, i \in \{1, \cdots, N \}$ and $1 \leq k \leq d$, there exists a positive constant $\alpha_c$ such that
\begin{equation} \label{D-3}
 v^k_{i}(t_2)\leq\bar{v}^k(t_1)- \alpha_c \frac{\D^k(t_1)}{2},
\end{equation}
then for all $t \geq t_2$,
\[ v^k_{i}(t)\leq\bar{v}^k(t_1)- \alpha_c e^{-n^{\infty} \kappa (t-t_2)}  \frac{\D^k(t_1)}{2}. \]
\end{lemma}
\begin{proof} We use the non-increasing property of $\bar{v}^k(t)$ in Lemma \ref{L4.1}, \eqref{C-0}, \eqref{D-1} and that
\[ \psi(\|\bx_j(t -\tau_{ij}(t)) - \bx_i(t) \|) \leq  \kappa \]
to see
\begin{align}
\begin{aligned} \label{D-4}
\dot{v}^k_{i}(t)&=\sum\limits_{j=1}^{N} \chi_{ij} \psi(\|\bx_j(t -\tau_{ij}(t)) - \bx_i(t) \|)
(v^k_{j}(t-\tau_{ij}(t))-v^k_{i}(t))\\
&\leq n^{\infty} \kappa (\bar{v}^k(t_1)-v^k_{i}(t)) =-n^{\infty} \kappa (v^k_{i}(t)-\bar{v}^k(t_1)), \quad t\geq t_2.
\end{aligned}
\end{align}
Then, we apply Gronwall's lemma for \eqref{D-4}, and use \eqref{D-3} to obtain
\[
v^k_{i}(t) \leq \bar{v}^k (t_1)+e^{-n^{\infty} \kappa (t-t_2)}(v^k_{i}(t_2)-\bar{v}^k(t_1)) \leq \bar{v}^k(t_1)- \alpha_c e^{-n^{\infty} \kappa (t-t_2)} \frac{\D^k (t_1)}{2}, \quad t \geq t_2.
\]
\end{proof}
In next proposition, we study how the velocity component of an agent is affected by the neighboring agents via the information flow.
\begin{proposition}\label{P4.1}
Suppose that the agent $i$ satisfies the hypothesis \eqref{D-3} of Lemma \ref{L4.2} and $i$ is a neighbor of $j$, i.e., $i\in\N_j$. Then,  we have
\begin{equation} \label{D-4-0}
v^k_{j}(t_2+\tau+1)\leq\bar{v}^k(t_1)-\frac{\underline{\psi} \alpha_c e^{-2n^{\infty} \kappa(\tau+1)}}{1+n^{\infty} \kappa}\frac{\D^k(t_1)}{2},
\end{equation}
where the positive constant $\underline{\psi}$ is such that $\psi(\|\bx_j(t -\tau_{ij}(t)) - \bx_i(t) \|) \geq\underline{\psi}$ for all $t_2\leq t\leq t_2+\tau+1$.
\end{proposition}
\begin{proof}
It follows from Lemma \ref{L4.2} that for $t \in [t_2, t_2 + \tau + 1]$,
\[ v^k_{i}(t)\leq \bar{v}^k(t_1)- \alpha_c e^{-n^{\infty} \kappa (t-t_2)}  \frac{\D^k(t_1)}{2} \leq  \bar{v}^k(t_1)- \alpha_c e^{-n^{\infty} \kappa (\tau + 1)}  \frac{\D^k(t_1)}{2} =: \Delta(t_1). \]
Let $i$ be a neighbor of $j$, i.e., $i\in\N_j$. Then, depending on the relative size of $v_j^k$ with $ \bar{v}^k(t_1)- \alpha_c e^{-n^{\infty} \kappa (\tau + 1)}  \frac{\D^k(t_1)}{2}$ in the time interval $[t_2,t_2+\tau+1]$, we consider the following two cases:
\begin{align}
\begin{aligned} \label{D-4-1}
& \mbox{Case A}:~\exists~ t^* \in [t_2, t_2 + \tau + 1] \quad \mbox{such that}~v^k_{j}(t^*)\leq \Delta(t_1). \\
& \mbox{Case B}:~\forall~t \in [t_2,t_2+\tau+1], \quad v^k_{j}(t) > \Delta(t_1).
\end{aligned}
\end{align}

\vspace{0.2cm}

\noindent $\bullet$~Case A: In this case, we use $\eqref{D-4-1}_1$ and the result of Lemma \ref{L4.2} to see that for all $ t\geq t^*$,
\begin{equation} \label{D-4-2}
 v^k_{j}(t)\leq\bar{v}^k(t_1)- \alpha_c e^{-n^{\infty} \kappa (t-t^*+\tau+1)} \frac{\D^k(t_1)}{2}.
\end{equation}
We choose $t = t_2+\tau+1$ in \eqref{D-4-2} and use $t_2 - t^* + 2(\tau + 1) \leq 2(\tau + 1)$ to get
\begin{equation}\label{D-5}
v^k_{j}(t_2+\tau+1)\leq\bar{v}^k (t_1)- \alpha_c e^{-2n^{\infty} \kappa (\tau+1)} \frac{\D^k(t_1)}{2}.
\end{equation}

\noindent $\bullet$~Case B: In this case, for all $t\in[t_2,t_2+\tau+1]$, we have
\[
v^k_{j}(t)>\bar{v}^k (t_1)- \alpha_c e^{-n^{\infty} \kappa (\tau+1)} \frac{\D^k(t_1)}{2}
\]
For $t\in[t_2+\tau,t_2+\tau+1]$, we use \eqref{D-1} and the monotonicity relation of ${\bar v}^k$ in Lemma \ref{L4.1}:
\[ v^k_{l}(t-\tau_{jl}(t)) \leq {\bar v}^k(t)  \leq  {\bar v}^k(t_2) \quad \mbox{and} \quad  {\bar v}^k(t_2) \leq  {\bar v}^k(t_1)  \]
to see
\begin{align}
\begin{aligned} \label{D-6}
\frac{d}{dt} v^k_{j}(t)&=\sum_{l\in\N_j}  \chi_{jl} \psi(\|\bx_l(t -\tau_{jl}(t)) - \bx_j(t) \|)
(v^k_{l}(t-\tau_{jl}(t))-v^k_{j}(t))\\
&=  \psi(\|\bx_i(t -\tau_{ji}(t)) - \bx_j(t) \|)
(v^k_{i}(t-\tau_{ji}(t))-v^k_{j}(t)) \\
& +  \sum_{\substack{l\in\N_j \\ l \not = i}}  \chi_{jl} \psi(\|\bx_l(t -\tau_{ij}(t)) - \bx_i(t) \|)
(v^k_{l}(t-\tau_{jl}(t))-v^k_{j}(t)) \\
&\leq  \psi(\|\bx_i(t -\tau_{ji}(t)) - \bx_j(t) \|)
(v^k_{i}(t-\tau_{ji}(t))-v^k_{j}(t)) +(n^{\infty}-1) \kappa (\bar{v}^k(t_2)-v^k_{j}(t))\\
&\leq  \psi(\|\bx_i(t -\tau_{ji}(t)) - \bx_j(t) \|) \left(\bar{v}^k(t_1)- \alpha_c e^{-n^{\infty} \kappa (\tau+1)} \frac{\D^k(t_1)}{2}-v^k_{j}(t)\right) \\
& +(n^{\infty}-1) \kappa ({\bar v}^k(t_2)-v^k_{j}(t))\\
&\leq \underline{\psi}\left(\bar{v}^k(t_1)- \alpha_c e^{-n^{\infty} \kappa (\tau+1)} \frac{\D^k (t_1)}{2}-v^k_{j}(t)\right) +(n^{\infty}-1)\kappa (\bar{v}^{k}(t_1)-v^k_{j}(t))\\
&=-\Big(\underline{\psi}+(n^{\infty}-1)\kappa \Big)
\Bigg[v^k_{j}(t) -\left(\bar{v}^k(t_1)-\frac{\underline{\psi} \alpha_c e^{-n^{\infty} \kappa (\tau+1)}}{\underline{\psi}+(n^{\infty}-1)\kappa}
\frac{\D^k(t_1)}{2}\right)\Bigg] \\
&=-\Big(\underline{\psi}+(n^{\infty}-1)\kappa \Big)
\Bigg[v^k_{j}(t) -\left(\bar{v}^k(t_1)-\frac{\underline{\psi} \alpha_c e^{-n^{\infty} \kappa (\tau+1)}}{1 + \underline{\psi}+(n^{\infty}-1)\kappa}
\frac{\D^k(t_1)}{2}\right)\Bigg] \\
&-\frac{\underline{\psi} \alpha_c e^{-n^{\infty} \kappa (\tau+1)}}{1 + \underline{\psi}+(n^{\infty}-1)\kappa}\frac{\D^k(t_1)}{2}
\end{aligned}
\end{align}
Next, we claim that there exists an instant $t^*\in[t_2+\tau,t_2+\tau+1]$ such that
\begin{equation}\label{D-7}
v^k_{j}(t^*)\leq\bar{v}^k(t_1)-\frac{\underline{\psi} \alpha_c e^{-n^{\infty} \kappa(\tau+1)}}{1+\underline{\psi}+(n^{\infty}-1)\kappa}
\frac{\D^k(t_1)}{2}.
\end{equation}
If this is not the case, that is, for all $t\in[t_2+\tau,t_2+\tau+1]$, we have
\begin{equation}\label{D-8}
v^k_{j}(t)>\bar{v}^k(t_1)-\frac{\underline{\psi} \alpha_c e^{-n^{\infty} \kappa (\tau+1)}}{1+\underline{\psi}+(n^{\infty}-1)\kappa}
\frac{\D^k(t_1)}{2}.
\end{equation}
Then it follows from \eqref{D-6} and \eqref{D-8} that
\begin{equation} \label{D-9}
\frac{d}{dt} v^k_{j}(t)\leq
-\frac{\underline{\psi} \alpha_c e^{-n^{\infty} \kappa (\tau+1)}}{1+\underline{\psi}+(n^{\infty}-1)\kappa}
\frac{\D^k(t_1)}{2}.
\end{equation}
We integrate the above relation \eqref{D-9} from $t_2+\tau$ to $t_2+\tau+1$ and use
\[ v^k_{j}(t_2 + \tau) \leq {\bar v}^k_{j}(t_2 + \tau) \leq \bar{v}^k(t_1) \]
to get
\[
v^k_{j}(t_2+\tau+1)\leq\bar{v}^k(t_1)-\frac{\underline{\psi} \alpha_c e^{-n^{\infty} \kappa (\tau+1)}}{1+\underline{\psi}+(n^{\infty}-1)\kappa}
\frac{\D^k(t_1)}{2},
\]
which gives a contradiction to \eqref{D-8}. Hence, we have \eqref{D-7}. Since $\underline{\psi} - \kappa < 0$, the relation \eqref{D-7} yields that there exists an instant $t^*\in[t_2+\tau,t_2+\tau+1]$ such that
\begin{equation} \label{D-10}
v^k_{j}(t^*) \leq \bar{v}^k(t_1)-\frac{\underline{\psi} \alpha_c e^{-n^{\infty} \kappa (\tau+1)}}{1+ n^{\infty} \kappa}
\frac{\D^k(t_1)}{2}.
\end{equation}
Finally,  we use Lemma \ref{L4.2} and \eqref{D-10} to obtain
\begin{equation}\label{D-11}
v^k_{j}(t_2+\tau+1)  \leq \bar{v}^k(t_1)-\frac{\underline{\psi} \alpha_c e^{-n^{\infty} \kappa ( t_2 + \tau + 1 - t^* + \tau+1)}}{1+ n^{\infty} \kappa}
\frac{\D^k(t_1)}{2}.
\end{equation}
We again use \eqref{D-11} and the relation
\[ e^{-n^{\infty} \kappa ( t_2 + \tau + 1 - t^* + \tau+1)} = \underbrace{e^{-n^{\infty} \kappa ( t_2- t^*)}}_{\geq 1}  e^{-2n^{\infty} \kappa (\tau + 1)} \geq e^{-2n^{\infty} \kappa (\tau + 1)} \]
 to derive the desired estimate \eqref{D-4-0}.
\end{proof}
We are now ready to provide a proof of Theorem \ref{T3.1} in the following subsection.
\subsection{Proof of Theorem~\ref{T3.1}} \label{sec:4.2}
We split its proof into two steps. First we show the exponential decay of $\D(t)$, and then using this exponential decay of $\D(t)$, we derive the uniform boundedness of the relative spatial positions. \newline

\noindent $\bullet$~Step A (Exponential decay of relative velocities):  Suppose that network topology ${\mathcal G},$ coupling strength $\kappa$, and time-delay $\tau$ satisfy the relations:
\[ \gamma_g > 0, \quad  n^{\infty} = \max_{1\leq i \leq N} |{\mathcal N}_i| > 0, \quad \tau > 0, \quad \kappa > 0. \]
For $\rho > 0$ and initial spatial condition $X(0)$, we define $\delta$ as follows.
\begin{equation*}
\delta := 1-\frac{ (\psi(X(0)+\rho))^{\gamma_g} e^{-n^{\infty} \kappa \gamma_g(3\tau+2)}}{2(1+n^{\infty} \kappa)^{\gamma_g}} > 0.
\end{equation*}
For given initial data set $\{ (\bx^0(t), \bv^0(t)):~-\tau \leq t \leq 0 \}$, from \eqref{C-1-1} there exists a positive constant $\rho>0$ such that
\begin{equation}\label{D-12}
2\sqrt{d}\gamma_g(2\tau+1)(1+ n^{\infty} \kappa)^{\gamma_g} \D(0)\leq
e^{-n^{\infty} \kappa {\gamma_g}(3\tau+2)} (\psi(X(0)+\rho))^{\gamma_g}  \rho.
\end{equation}
Since $\D(t)$ is non-increasing in $t$, we will show that for the exponential decay of relative velocities it is enough to have
\begin{equation} \label{D-13}
\D(n\gamma_g(2\tau+1))\leq\delta^n\D(0), \quad n \in \NN \cup \{0 \}.
\end{equation}

\bigskip

Next, we verify the relation \eqref{D-13} using mathematical induction on $n$. \newline

\noindent $\diamond$~Step A (Initial step):  Note that  the relation \eqref{D-13} is trivially true.  \newline

\noindent $\diamond$~Step B (induction step): Suppose that the relation \eqref{D-13} holds for all $n = 0, 1, \cdots, T \in \NN \cup \{0 \}$, i.e.,
\begin{equation}\label{D-14}
\D(n\gamma_g(2\tau+1))\leq\delta^n\D(0), \quad n = 0, \cdots, T.
\end{equation}
Let $t \in [T\gamma_g(2\tau+1), (T+1)\gamma_g(2\tau+1)]$. Then, it follows from $\eqref{C-1}_1$ that we have
\begin{align}
\begin{aligned}\label{D-15}
&\|\bx_i(t)-\bx_j(t-\tau_{ij}(t))\|\\
&\hspace{1cm} =\|\bx_i(0)-\bx_j(-\tau_{ij}(t))\|
+\int_0^t\frac{d}{ds}\|\bx_i(s)-\bx_j(s-\tau_{ij}(t))\|ds\\
&\hspace{1cm} \leq\|\bx_i(0)-\bx_j(-\tau_{ij}(t))\|+\int_0^t\|\bv_i(s)-\bv_j(s-\tau_{ij}(t))\|ds\\
&\hspace{1cm} \leq X(0)
+\sum_{n=0}^T\int_{n\gamma_g(2\tau+1)}^{(n+1)\gamma_g (2\tau+1)}\|\bv_i(s)-\bv_j(s-\tau_{ij}(t))\|ds.\\
\end{aligned}
\end{align}
On the other hand, note that
\begin{align}
\begin{aligned} \label{D-16}
&\|\bv_i(s)-\bv_j(s-\tau_{ij}(t))\|\\
&\hspace{0.5cm} \leq\sqrt{d}\max_{1\leq k \leq d}|v^k_{i}(s)-v^k_{j}(s-\tau_{ij}(t))| \leq\sqrt{d}\max_{1\leq k\leq d}(\bar{v}^k(s)-\underline{v}_k(s)) =\sqrt{d}\D(s).
\end{aligned}
\end{align}
We combine \eqref{D-15} and \eqref{D-16} to obtain
\begin{align}
\begin{aligned} \label{D-17}
&\|\bx_i(t)-\bx_j(t-\tau_{ij}(t))\|\\
&\hspace{0.5cm} \leq X(0)
+\sqrt{d}\sum_{n=0}^T\int_{n\gamma_g(2\tau+1)}^{(n+1)\gamma_g(2\tau+1)}\D(s)ds  \\
& \hspace{0.5cm} \leq X(0)
+\sqrt{d}\gamma_g(2\tau+1)\sum_{n=0}^T\D(n\gamma_g(2\tau+1))\\
&\hspace{0.5cm} \leq X(0) +\sqrt{d}\gamma_g(2\tau+1)\sum_{n=0}^T\delta^n\D(0) \\
& \hspace{0.5cm} \leq X(0)+\sqrt{d}\gamma_g(2\tau+1)
\frac{\D(0)}{1-\delta}\\
& \hspace{0.5cm} \leq X(0)+\rho.
\end{aligned}
\end{align}
Here, the second inequality follows from the fact that $\D(t)$ is non-increasing with $t$, and the last inequality is from \eqref{D-12}:
\[ \sqrt{d}\gamma_g(2\tau+1) \D(0) \leq  \frac{e^{-n^{\infty} \kappa {\gamma_g}(3\tau+2)}  (\psi(X(0)+\rho))^{\gamma_g}}{2(1+ n^{\infty} \kappa)^{\gamma_g}} \rho = (1-\delta) \rho.   \]
Next, for $j \in {\mathcal N}_i$, we use \eqref{D-17} and monotonicity of $\psi$ to obtain
\[ \chi_{ij} \psi(\|\bx_j(t -\tau_{ij}(t)) - \bx_i(t) \|) = \psi(\|\bx_j(t -\tau_{ij}(t)) - \bx_i(t) \|) \geq \psi(X(0)+\rho). \]

On the other hand, since $\G$ has a spanning tree, we take a root $r$ of $\G$ such that
\[ \gamma_g =\max_{j\in\V}\dist(r,j). \]
For a fixed $k \in \{1,\cdots, d\}$, we consider the following two cases:
\begin{align*}
\begin{aligned}
& \mbox{Case (i)}:~v^k_{r}(T\gamma_g (2\tau+1))\leq\frac{\bar{v}^k (T\gamma_g (2\tau+1))
+\underline{v}^k(T\gamma_g(2\tau+1))}{2}. \\
& \mbox{Case (ii)}:~v^k_{r}(T\gamma_g (2\tau+1))\geq\frac{\bar{v}^k (T\gamma_g (2\tau+1))
+\underline{v}^k(T\gamma_g(2\tau+1))}{2}.
\end{aligned}
\end{align*}
(i) We first consider the case:
\begin{align}
\begin{aligned}  \label{D-18}
v^k_{r}(T\gamma_g (2\tau+1))&\leq\frac{\bar{v}^k (T\gamma_g (2\tau+1))
+\underline{v}^k(T\gamma_g(2\tau+1))}{2} \\
&=\bar{v}^k(T\gamma_g(2\tau+1))-\frac{\D^k(T\gamma_g(2\tau+1))}{2}.
\end{aligned}
\end{align}

We apply Lemma \ref{L4.2} with $\alpha_c = 1$ to find that for all $t\geq T\gamma_g (2\tau+1)$,
\begin{equation*}
v^k_{r}(t)\leq\bar{v}^k(T\gamma_g(2\tau+1))-
 e^{-n^{\infty} \kappa (t-T\gamma_g(2\tau+1))}\frac{\D^k(T\gamma_g(2\tau+1))}{2}.
\end{equation*}
For any $i\neq r$, we can find a path from $r$ to $i$ of length $p\leq \gamma_g$:
\[  r=i_0~\rightarrow~i_1~\rightarrow~\ldots~\rightarrow i_p=i. \]
\noindent For $i_1$, we apply Proposition \ref{P4.1} with
\[ i=r, \quad j=i_1, \quad t_1=t_2=T\gamma_g (2\tau+1), \quad  \alpha_c=1 \quad \mbox{and} \quad \underline{\psi}= \psi(X(0)+\rho) \]
to get
\[
v_{i_1}^k(T\gamma_g(2\tau+1)+\tau+1)\leq\bar{v}^k(T\gamma_g(2\tau+1))-
 \frac{\psi(X(0)+\rho) e^{-2n^{\infty} \kappa(\tau+1)}}{1+n^{\infty} \kappa}\frac{\D^k(T\gamma_g(2\tau+1))}{2}.
\]
We again apply Lemma \ref{L4.2} to the above relation to find
\[
v_{i_1}^k(T\gamma_g(2\tau+1)+2\tau+1)\leq\bar{v}^k(T\gamma_g(2\tau+1))-
 \frac{\psi(X(0)+\rho) e^{-n^{\infty} \kappa(3\tau+2)}}{1+n^{\infty} \kappa}\frac{\D^k(T\gamma_g (2\tau+1))}{2}.
\]
We continue to apply Proposition \ref{P4.1} with
\begin{align*}
\begin{aligned}
& i=i_1, \quad j=i_2, \quad t_1=T\gamma_g(2\tau+1), \quad t_2=T\gamma_g(2\tau+1)+2\tau+1, \\
&  \alpha_c= \frac{\psi(X(0)+\rho) e^{-n^{\infty}\kappa(3\tau+2)}}{1+ n^{\infty} \kappa} \quad \mbox{and} \quad \underline{\psi}= \psi(X(0)+\rho)
\end{aligned}
\end{align*}
to get
\begin{align*}
\begin{aligned}
& v_{i_2}^k(T\gamma_g(2\tau+1)+2\tau+1+\tau+1) \\
& \hspace{1cm} \leq\bar{v}^k(T\gamma_g(2\tau+1))-
 \frac{(\psi(X(0)+\rho))^2 e^{-n^{\infty} \kappa(5\tau+4)}}{(1+n^{\infty} \kappa)^2}\frac{\D^k(T\gamma_g(2\tau+1))}{2}.
\end{aligned}
\end{align*}
We proceed the above argument until $i_p=i$ to derive
\begin{align}
\begin{aligned} \label{D-20}
&v_{i}^k(T\gamma_g (2\tau+1)+(p-1)(2\tau+1)+\tau+1)\\
&\hspace{1cm} \leq\bar{v}^k(T\gamma_g(2\tau+1))- \frac{ (\psi(X(0)+\rho))^p e^{-n^{\infty} \kappa((3(p-1)+2)\tau+2p)}}{(1+n^{\infty}\kappa)^p}\frac{\D^k(T\gamma_g(2\tau+1))}{2}.
\end{aligned}
\end{align}
Then, \eqref{D-20} and Lemma \ref{L4.2} yield
\begin{align*}
\begin{aligned} \label{D-21}
&v_{i}^k(T\gamma_g(2\tau+1)+(\gamma_g-1)(2\tau+1)+\tau+1)\\
&\hspace{1cm} \leq\bar{v}^k(T\gamma_g(2\tau+1))- \frac{(\psi(X(0)+\rho))^p  e^{-n^{\infty} \kappa(\gamma_g(2\tau+1)+(p-1)\tau+p)}}{(1+n^{\infty}\kappa)^p}\frac{\D^k(T\gamma_g(2\tau+1))}{2}\\
&\hspace{1cm} \leq\bar{v}^k(T\gamma_g(2\tau+1))- \frac{(\psi(X(0)+\rho))^{\gamma_g} e^{-n^{\infty} \kappa((3\gamma_g-1)\tau+2\gamma_g)}}{(1+n^{\infty}\kappa)^{\gamma_g}}\frac{\D^k(T\gamma_g(2\tau+1))}{2}.
\end{aligned}
\end{align*}
We recall \eqref{D-18} to see that the above inequality holds for all $1\leq i\leq N$. We once again apply Lemma \ref{L4.2} to conclude that for all $t\geq T\gamma_g(2\tau+1)+(\gamma_g-1)(2\tau+1)+\tau+1$ and $i\leq N$:
\[
v^k_{i}(t)\leq \bar{v}^k(T\gamma_g(2\tau+1))- \frac{(\psi(X(0)+\rho))^{\gamma_g} e^{-n^{\infty}\kappa(t-T\gamma_g(2\tau+1)+\gamma_g(\tau+1))}}{(1+n^{\infty} \kappa)^{\gamma_g}}\frac{\D^k(T\gamma_g(2\tau+1))}{2}.
\]
Therefore, we have
\begin{align}
\begin{aligned} \label{D-22}
&\bar{v}^k((T+1)\gamma_g(2\tau+1))\\
&\hspace{1cm} =\max_{\substack{1\leq i\leq k\\(T+1)\gamma_g(2\tau+1)-\tau\leq t\leq(T+1)\gamma_g(2\tau+1)}}v^k_{i}(t)\\
&\hspace{1cm} \leq\bar{v}^k(T\gamma_g(2\tau+1)) -\frac{(\psi(X(0)+\rho))^{\gamma_g} e^{-n^{\infty} \kappa \gamma_g(3\tau+2)}}{(1+n^{\infty}\kappa)^{\gamma_g}} \frac{\D^k(T\gamma_g(2\tau+1))}{2}.
\end{aligned}
\end{align}
Again, we use the induction hypothesis \eqref{D-14} and \eqref{D-22} to get
\begin{align*}
\begin{aligned}
&\D^k((T+1)\gamma_g(2\tau+1))\\
&\hspace{0.2cm} =\bar{v}^k((T+1)\gamma_g(2\tau+1))-\underline{v}^k((T+1)\gamma_g(2\tau+1))\\
&\hspace{0.2cm} \leq\bar{v}^k(T\gamma_g(2\tau+1))
-\frac{(\psi(X(0)+\rho))^{\gamma_g} e^{-n^{\infty} \kappa \gamma_g(3\tau+2)}}{(1+ n^{\infty} \kappa)^{\gamma_g}}
\frac{\D^k(T\gamma_g(2\tau+1))}{2}-\underline{v}^k(T\gamma_g(2\tau+1))\\
&\hspace{0.2cm} =\delta\D^k(T\gamma_g(2\tau+1)) \leq\delta^{T+1}\D(0),
\end{aligned}
\end{align*}
which yields the claim \eqref{D-13}.  \newline
(ii) For the other case, i.e., 
\begin{align}
\begin{aligned}\label{D-19}
v^k_{r}(T\gamma_g (2\tau+1))
&\geq\bar{v}^k(T\gamma_g(2\tau+1))-\frac{\D^k(T\gamma_g(2\tau+1))}{2}\\
&=\underline{v}^k(T\gamma_g(2\tau+1))+\frac{\D^k(T\gamma_g(2\tau+1))}{2},
\end{aligned}
\end{align}
we, in turn, consider the vector $-v^k(T\gamma_g (2\tau+1))=\left(-v^k_1(T\gamma_g (2\tau+1)),\ldots, -v^k_N(T\gamma_g (2\tau+1))\right)$. Note that the diameter of the vector $-v^k(T\gamma_g (2\tau+1))$ is the same with that of $v^k(T\gamma_g (2\tau+1))$. It follows from \eqref{D-19} that
\begin{align*}
\begin{aligned}
-v^k_{r}(T\gamma_g (2\tau+1))
&\leq-\underline{v}^k(T\gamma_g(2\tau+1))-\frac{\D^k(T\gamma_g(2\tau+1))}{2}\\
&=\overline{(-v)}^k(T\gamma_g(2\tau+1))-\frac{\D^k(T\gamma_g(2\tau+1))}{2}.\\
\end{aligned}
\end{align*}
Then we proceed the following arguments by observing that Lemma \ref{L4.2} and Proposition \ref{P4.1} still hold with $v_i^k(t)$ replaced by $-v_i^k(t)$.\newline

Note that the relation \eqref{D-13} holds only for integer $n$. However, using the monotonic decreasing property of $\D(t)$, we can easily see
\[ \D^k(t \gamma_g(2\tau+1)) \leq \delta^{[t]} \D(0), \quad t \geq 0. \]
We set
\[ s := t\gamma_g(2\tau+1). \]
Then the relation \eqref{D-13} yields
\begin{equation} \label{D-23}
 |v^k_i(s) - v^k_j(s)|  \leq\D(s)  \leq \delta^{[\frac{s}{\gamma_g (2\tau +1)}]} \D(0).
\end{equation}
Again, we use the above relation and the fact $\delta < 1$ to find
\begin{equation} \label{D-24}
\lim_{s \to \infty} ||\bv_i(s) -\bv_j(s)|| \leq \sqrt{d}  \D(0) \lim_{s \to \infty} \delta^{[\frac{s}{\gamma_g (2\tau +1)}]} = 0.
\end{equation}

\vspace{0.2cm}

\noindent $\bullet$~Step B (Uniform boundedness of relative positions): We use \eqref{D-23} to obain the uniform boundedness of spatial differences:
\begin{align*}
\begin{aligned}
& |x^k_i(t)- x^k_j(t)|\\
&\hspace{1cm} \leq |x^k_i(0)-x^k_j(0)| +\int_0^t   |v^k_i(s)- v^k_j(s)| ds\\
&\hspace{1cm} \leq  |x^k_i(0)-x^k_j(0)|  + \D(0) \int_0^t \delta^{[\frac{s}{\gamma_g (2\tau +1)}]} ds \\
&\hspace{1cm} \leq  |x^k_i(0)-x^k_j(0)|  + \D(0) \delta^{-1} \int_0^t \delta^{\frac{s}{\gamma_g (2\tau +1)}} ds \\
& \hspace{1cm} \leq |x^k_i(0)-x^k_j(0)|  + \D(0) \frac{\delta^{-1} \gamma_g (2\tau +1)}{\ln \delta^{-1}} \Big(1 - \delta^{\frac{t}{\gamma_g (2\tau +1)}} \Big).
\end{aligned}
\end{align*}
This again yields
\begin{equation} \label{D-25}
 ||\bx_i(t) - \bx_j(t) \| \leq \sum_{k=1}^{d}|x^k_i(0)-x^k_j(0)|  + d \D(0) \frac{\delta^{-1} \gamma_g (2\tau +1)}{\ln \delta^{-1}}.
 \end{equation}
Finally, we combine \eqref{D-24} and \eqref{D-25} to complete the proof of Theorem \ref{T3.1}.

\subsection{Proof of Corollary \ref{C3.1}} \label{sec:4.3} In this subsection, we provide sufficient conditions for initial data leading to the mono-cluster flocking depending on the decay rate of the C-S communication weight.  Note that the sufficient condition \eqref{C-1-1} is equivalent to the following relation:
\[  \D(0) <   \frac{C_\infty \kappa^{\gamma_g} \rho}{[ 1 + (X(0) + \rho)^2 ]^{\beta \gamma_g}} =: R(\rho). \]
\begin{lemma} \label{L4.3}
Let  $\{(\bx_i(t),\bv_i(t))\}_{i=1}^N$ be a solution to system~\eqref{C-1}, and $\D$ and $X(0)$ be quantities defined in \eqref{A-3}. Then, we have the following estimates:
\begin{enumerate}
\item
Suppose that $\beta  \in [0, \frac{1}{2\gamma_g})$. Then for any initial data with $\D(0) < \infty$, there exists a positive constant $\rho_*$ such that if $\rho > \rho_*$, we have
\[  \D(0) <   \frac{C_\infty \kappa^{\gamma_g} \rho}{[ 1 + (X(0) + \rho)^2 ]^{\beta \gamma_g}}.     \]
\item
Suppose that $\beta = \frac{1}{2\gamma_g}$. Then for any initial data  with $\D(0) < C_\infty \kappa^{\gamma_g}$, there exists a positive constant $\rho_*$ such that if $\rho > \rho_*$, we have
\[  \D(0) <   \frac{C_\infty \kappa^{\gamma_g} \rho}{[ 1 + (X(0) + \rho)^2 ]^{\beta \gamma_g}}.     \]
\item
Suppose that $\beta \in (\frac{1}{2\gamma_g}, \infty)$. Then we have 
\[ \sup_{0<\rho<\infty}R( \rho)=R(\rho^+) \]
where $\rho_+$ is a positive constant such that
\[ \rho_+ := \frac{ X(0) (1-\beta \gamma_g) + \sqrt{  \beta^2 \gamma_g^2 X(0)^2 + 2 \beta \gamma_g -1}}{2\beta \gamma_g -1}.     \]
\end{enumerate}
\end{lemma}
\begin{proof} First, we consider the function:
\[  R(\rho) := \frac{C_\infty \kappa^{\gamma_g} \rho}{[1 + (X(0) + \rho)^2]^{\beta \gamma_g}}, \quad \rho > 0.   \]
By direct calculation, we have
\begin{equation} \label{D-11-1}
 R^{\prime}(\rho) = C_\infty \kappa^{\gamma_g} \Big( \frac{(1 - 2\beta \gamma_g) \rho^2 + 2X(0) (1-\beta \gamma_g) \rho + 1 + X(0)^2}{[ 1 + (X(0) + \rho)^2]^{\beta \gamma_g + 1}} \Big ).
 \end{equation}

\vspace{0.5cm}

\noindent(i)~Suppose that $0 \leq \beta < \frac{1}{2\gamma_g}$. It is easy to see that
\begin{equation} \label{D-11-2}
 R(\rho) \approx C_\infty \kappa^{\gamma_g} \rho^{1-2\beta \gamma_g} \quad \mbox{and} \quad  R^{\prime}(\rho) \approx  R(\rho) \frac{(1-2\beta \gamma_g)}{\rho}, \quad \rho \gg 1.
 \end{equation}
The relations \eqref{D-11-1} and \eqref{D-11-2} yield that $R(\rho)$  monotonically increases to $\infty$, as $\rho \to \infty$. Thus, we can choose $\rho_* > 0$ such that
\[ \D(0) < R(\rho_*) \leq R(\rho), \qquad \forall~\rho > \rho_*. \]

\vspace{0.5cm}

\noindent (ii)~Suppose that $\beta = \frac{1}{2\gamma_g}$. In this case, we have
\[   R(\rho) =  \frac{C_\infty \kappa^{\gamma_g} \rho}{[1 + (X(0) + \rho)^2]^{\frac{1}{2}}}, \qquad   R^{\prime}(\rho) = R(\rho) \frac{(X(0) \rho + 1 + X(0)^2)}{\rho [ 1 + (X(0) + \rho)^2]} > 0. \]
Thus, $R(\rho)$ monotonically increases to $C_\infty$, as $\rho \to \infty$. Thus, there exists $\rho_* > 0$ such that
\[ \D(0) < R(\rho_*) \leq R(\rho) \leq C_\infty \kappa^{\gamma_g}, \qquad \forall~\rho > \rho_*. \]

\vspace{0.5cm}

\noindent (iii)~Suppose that $\beta > \frac{1}{2\gamma_g}$. We set a quadratic polynominal $h$:
\[  h(\rho):= (1 - 2\beta \gamma_g) \rho^2 + 2X(0) (1-\beta \gamma_g) \rho + 1 + X(0)^2 = 0. \]
Then, the discriminant $\Delta$ of $h = 0$ becomes
\begin{align*}
\begin{aligned}
\Delta &=  4 X(0)^2 (1-\beta \gamma_g)^2 - 4 (1 - 2\beta \gamma_g) ( 1 + X(0)^2 )  \\
&= 4 \Big(  \beta^2 \gamma_g^2 X(0)^2 + 2 \beta \gamma_g -1 \Big) > 0.
\end{aligned}
\end{align*}
and two roots $\rho_\pm$ of $h(\rho) = 0$ become
\[ \rho_{\pm} := \frac{ -X(0) (1-\beta \gamma_g) \pm \sqrt{  \beta^2 \gamma_g^2 X(0)^2 + 2 \beta \gamma_g -1}}{(1 - 2\beta \gamma_g)}.     \]
Note that
\[
R^{\prime}(\rho) = 0,~~ \rho > 0 \quad \Longleftrightarrow \quad h(\rho) = 0,~~\rho > 0 \quad \Longleftrightarrow \quad \rho = \rho_+.
\]
Then, we have
\begin{equation*}
 \sup_{0\leq \rho < \infty} R(\rho) = R(\rho_+).
\end{equation*}
\end{proof}

\noindent {\it Proof of Corollary \ref{C3.1}}: Depending on the relative sizes of $\beta$, we consider three cases:
\[ 2\gamma_g \beta  < 1, \qquad 2\gamma_g \beta = 1, \qquad 2\gamma_g \beta  > 1. \]
The last case is trivial, so we only consider the first two cases. (Note that Lemma \ref{L4.3} (3) helps us derive the best upper bound for $\D(\rho)$ from Theorem \ref{T3.1}.)

\noindent $\diamond$ Case 1~$(\beta<\frac{1}{2\gamma_g})$: In this case, it follows from (1) in Lemma \ref{L4.3} that there is no restriction on the initial data to guarantee the relation \eqref{D-23}.  \newline

\noindent $\diamond$ Case 2~$(\beta = \frac{1}{2\gamma_g})$: In this case, if initial diameter $\D(0)$ satisfies
\[  \D(0) <   \frac{e^{-n^{\infty} \kappa {\gamma_g}(3\tau+2)} \kappa^{\gamma_g} }{  2\sqrt{d}\gamma_g(2\tau+1)(1+ n^{\infty} \kappa)^{\gamma_g}   } = C_\infty \kappa^{\gamma_g}, \]
then we have the relation \eqref{D-23}.  \newline

This completes the proof of Corollary \ref{C3.1}.

\section{The discrete Cucker-Smale model}\label{sec:5}
\setcounter{equation}{0}
In this section, we provide a proof of our second main result Theorem \ref{T3.2} for the discrete C-S model on a digraph. The presentation will be almost parallel to Section \ref{sec:4}. \newline

Note that the $k$-th component $(x_i^k, v_i^k)$ satisfies the difference equations:
\begin{align}
\begin{aligned} \label{F-1}
x^k_{i}[t+1] &=x^k_{i}[t]+hv^k_{i}[t], \quad t \in \NN \cup \{0 \},\\
v^k_{i}[t+1] &=v^k_{i}[t]+ h \sum_{j=1}^{N} \chi_{ij} \psi(\|\bx_j[t -\tau_{ij}[t]] - \bx_i[t] \|)
(v^k_{j}[t-\tau_{ij}[t]]-v^k_{i}[t]).
\end{aligned}
\end{align}
Here $\tau_{ij}[t]\in\NN \cup \{0 \}$ and $0\leq\tau_{ij}[t]\leq \tau$ for all $t\in\NN \cup \{0 \}$.

\subsection{Basic estimates} \label{sec:5.1} Next, we provide several estimates to be crucially used in the proof of Theorem \ref{T3.2}.

\begin{lemma}\label{L5.1}
Suppose that time step $h$ and coupling strength $\kappa$ satisfy
\[ \kappa > 0, \quad  0<h \kappa<\frac{1}{n^{\infty}}, \]
and let $\{(\bx_i[t],\bv_i[t])\}_{i=1}^N$ be a solution to system~\eqref{F-1}. Then, $\bar{v}^k[t]$ and $\underline{v}^k[t]$ defined in \eqref{A-3} are nonincreasing and nondecreasing, respectively.
\end{lemma}
\begin{proof}
For $k \in \{1, \cdots, N \}$, we first show that $\bar{v}^k[t]$ is nonincreasing. The estimate for $\underline{v}^k[t]$ can be treated similarly. First recall definition of $\bar{v}^k[t]$:
\[ \bar{v}^k[t]=\max_{\substack{1\leq i\leq k\\t-\tau\leq s\leq t}}v^k_{i}[s]. \]
We claim:
\[ \label{F-1-1}
 \bar{v}^k[t+1]\leq\bar{v}^k[t] \quad \mbox{for all}~~t \in \NN.
\]
For $\bar{v}^k[t+1]$, we can find $i^*$ and $t^*\in [t+1-\tau, t+1]$ such that $\bar{v}^k[t+1]=v^k_{i^*}[t^*]$. We now need to consider two cases depending on the value of $t^*$:
\[ t^*\in[t+1-\tau, t], \quad t^* = t + 1. \]

\vspace{0.2cm}

\noindent $\bullet$ Case A ($t^*\in[t+1-\tau, t]$): In this case, it follows from definition of $\bar{v}^k[t]$ that
\[ v^k_{i^*}[t^*]\leq\bar{v}^k[t], \quad \mbox{i.e.,} \quad \bar{v}^k[t+1]=v^k_{i^*}[t^*]\leq\bar{v}^k[t]. \]

\vspace{0.2cm}

\noindent $\bullet$ Case B ($t^*=t+1$): In this case, we rewrite $\eqref{F-1}_2$ at $i^*$ as follows.
\begin{align*}
\begin{aligned}
v^k_{i^*}[t+1] &=\left(1- h\sum_{j=1}^{N} \chi_{i^*j} \psi(\|\bx_j[t -\tau_{i^*j}[t]] - \bx_i^*[t] \|)
\right)v^k_{i^*}[t] \\
&+ h \sum_{j=1}^{N} \chi_{i^*j} \psi(\|\bx_j[t -\tau_{i^*j}[t]] - \bx_i^*[t] \|) v^k_{j}[t-\tau_{i^*j}[t]]\\
&\leq\left(1-h \sum_{j=1}^{N} \chi_{i^*j} \psi(\|\bx_j[t -\tau_{i^*j}[t]] - \bx_i^*[t] \|)   \right)\bar{v}^k[t] \\
&+ h\sum_{j =1}^{N} \chi_{i^*j} \psi(\|\bx_j[t -\tau_{i^*j}[t]] - \bx_i^*[t] \|) \bar{v}^k[t]\\
&=\bar{v}^k[t],
\end{aligned}
\end{align*}
where we use the fact that
\[ 1- h \sum_{j=1}^{N} \chi_{i^*j} \psi(\|\bx_j[t -\tau_{i^*j}[t]] - \bx_i^*[t] \|) \geq1-hn^{\infty} \kappa >0. \]
Therefore, we can conclude that $\bar{v}^k[t]$ is non-increasing. The monotonicity of $\underline{v}^k[t]$ can be shown similarly.
\end{proof}
\begin{lemma}\label{L5.2}
Suppose that time step $h$ and coupling strength $\kappa$ satisfy
\[ 0<h \kappa<\frac{1}{n^{\infty}}, \]
and let $\{(\bx_i[t],\bv_i[t])\}_{i=1}^N$ be a solution to system~\eqref{F-1}. If for $t_2\geq t_1\geq0$, there exists a positive constant $\alpha_d$ such that
\[
v^k_{i}[t_2]\leq\overline{v}^k[t_1]- \alpha_d \frac{\D^k[t_1]}{2},
\]
then, for all $t\geq t_2$,
\[
v^k_{i}[t]\leq\overline{v}^k[t_1]-(1-hn^{\infty} \kappa)^{t-t_2} \alpha_d \frac{\D^k[t_1]}{2}.
\]
\end{lemma}
\begin{proof}
We use $\eqref{F-1}_2$ to see
\begin{align*}
\begin{aligned}
&v^k_{i}[t_2+1]\\
&\hspace{0.5cm} =v^k_{i}[t_2]+ h \sum_{j=1}^{N} \chi_{ij}   \psi(\|\bx_j[t -\tau_{ij}[t]] - \bx_i[t] \|)
(v^k_{j}[t_2-\tau_{ij}[t_2]]-v^k_{i}[t_2])\\
&\hspace{0.5cm} \leq v^k_{i}[t_2]+hn^{\infty} \kappa(\bar{v}^k[t_1]-v^k_{i}[t_2]) =(1-hn^{\infty} \kappa)v^k_{i}[t_2]
+hn^{\infty} \kappa \bar{v}^k[t_1]\\
&\hspace{0.5cm} \leq(1-hn^{\infty} \kappa)\left(\bar{v}^k[t_1]-\alpha_d \frac{\D^k[t_1]}{2}\right)
+h n^{\infty} \kappa \bar{v}^k[t_1]\\
&\hspace{0.5cm}=\bar{v}^k[t_1]-(1-hn^{\infty}\kappa ) \alpha_d \frac{\D^k[t_1]}{2}.
\end{aligned}
\end{align*}
Inductively, we obtain that for all $t\geq t_2$,
\[
v^k_{i}[t]\leq\bar{v}^k[t_1]-(1-hn^{\infty} \kappa)^{t-t_2} \alpha_d \frac{\D^k[t_1]}{2}.
\]
\end{proof}
\begin{proposition}\label{P5.1}
Suppose that the agent $i$ satisfies the assumptions of Lemma \ref{L5.2} and  $i$ is a neighbor of $j$, i.e., $i\in\N_j$. Then,  we have
\[ v^k_{j}[t_2+\tau+1]
\leq\bar{v}^k[t_1]-\frac{h\underline{\psi} \alpha_d (1-hn^{\infty} \kappa)^{2\tau+1} }{1+n^{\infty} \kappa}
\frac{\D^k[t_1]}{2},
\]
where the positive constant $\underline{\psi}$ is such that $ \chi_{ij} \psi(\|\bx_j[t -\tau_{ij}[(t]] - \bx_i[t] \|) \geq\underline{\psi}$ for all $t_2\leq t\leq t_2+\tau+1$.
\end{proposition}
\begin{proof}
We consider the following two cases. \newline

\noindent $\bullet$ Case A: If there exists some time instant $t^*\in[t_2,t_2+\tau+1]$ satisfying
\[ v^k_{j}[t^*]\leq\bar{v}^k[t_1]-(1-hn^{\infty} \kappa)^\tau \alpha_d \frac{\D^k[t_1]}{2}, \]
then Lemma \ref{L5.2} implies
\begin{equation*}
v^k_{j}[t_2+\tau+1]\leq\bar{v}^k[t_1]-(1-hn^{\infty} \kappa)^{2\tau+1} \alpha_d \frac{\D^k[t_1]}{2}.
\end{equation*}

\vspace{0.2cm}

\noindent $\bullet$ Case B: For all $t\in[t_2, t_2+\tau+1]$, we have
\[ v^k_{j}[t]>\bar{v}^k[t_1]-(1-hn^{\infty} \kappa)^\tau \alpha_d \frac{\D^k[t_1]}{2}. \]
By our hypothesis on $i$, we have
\begin{align}
\begin{aligned} \label{F-3}
&v^k_{j}[t_2+\tau+1]\\
&\hspace{0.5cm} =v^k_{j}[t_2+\tau] + h \sum_{l = 1}^{N}  \chi_{jl}   \psi(\|\bx_l[t_2+\tau -\tau_{jl}[t_2 + \tau]] - \bx_j[t_2 + \tau] \|) \\
& \hspace{0.5cm} \times (v^k_{l}[t_2+\tau-\tau_{jl}[t_2+\tau]]-v^k_{j}[t_2+\tau])\\
&\hspace{0.5cm}\leq v^k_{j}[t_2+\tau]
+h  \psi(\|\bx_i[t_2+\tau -\tau_{ji}[t_2 + \tau]] - \bx_j[t_2 + \tau] \|) \\
& \hspace{0.5cm} \times
(v^k_{i}[t_2+\tau-\tau_{ji}[t_2+\tau]]-v^k_{j}[t_2+\tau])\\
&\hspace{0.5cm}+(n^{\infty}-1)h \kappa (\bar{v}^k[t_1]-v^k_{j}[t_2+\tau])\\
&\hspace{0.5cm} \leq v^k_{j}[t_2+\tau]
+h\underline{\psi}\left(\bar{v}^k[t_1]-(1-hn^{\infty} \kappa)^\tau
\alpha_d \frac{\D^k[t_1]}{2}-v^k_{j}[t_2+\tau]\right)\\
&\hspace{0.5cm} +h(n^{\infty}-1)\kappa (\bar{v}^k[t_1]-v^k_{j}[t_2+\tau])\\
&\hspace{0.5cm} =v^k_{j}[t_2+\tau]-h(\underline{\psi}+(n^{\infty}-1)\kappa)v^k_{j}[t_2+\tau]\\
&\hspace{0.5cm}+h(\underline{\psi}+(n^{\infty}-1)\kappa)\bar{v}^k[t_1]
-h\underline{\psi}(1-h n^{\infty} \kappa)^\tau \alpha_d \frac{\D^k[t_1]}{2}.
\end{aligned}
\end{align}
We now distinguish two subcases. \newline

\noindent $\diamond$~Case B.1: If
\[
v^k_{j}[t_2+\tau]>\bar{v}^k[t_1]-\frac{\underline{\psi}(1-hn^{\infty} \kappa)^\tau \alpha_d}
{1+\underline{\psi}+(n^{\infty}-1)\kappa} \frac{\D^k[t_1]}{2},
\]
then \eqref{F-3} gives
\begin{align*}
\begin{aligned}
&v^k_{j}[t_2+\tau+1]\\
&\hspace{0.5cm} <\bar{v}^k[t_1]
-h(\underline{\psi}+(n^{\infty}-1)\kappa)
\left(\bar{v}^k[t_1]-\frac{\underline{\psi}(1-hn^{\infty} \kappa)^\tau \alpha_d}
{1+\underline{\psi}+(n^{\infty}-1)\kappa} \frac{\D^k[t_1]}{2} \right)\\
&\hspace{0.5cm} +h(\underline{\psi}+(n^{\infty}-1)\kappa)\bar{v}^k[t_1]
-h\underline{\psi}(1-h n^{\infty} \kappa)^\tau \alpha_d \frac{\D^k[t_1]}{2}\\
&\hspace{0.5cm} =\bar{v}^k[t_1]-\frac{h\underline{\psi}(1-hn^{\infty} \kappa)^\tau \alpha_d}{1+\underline{\psi}+(n^{\infty}-1)\kappa}
\frac{\D^k[t_1]}{2}.
\end{aligned}
\end{align*}

\vspace{0.5cm}

\noindent $\diamond$~Case B.2: Otherwise, \eqref{F-3} again gives
\begin{align*}
\begin{aligned}
&v^k_{j}[t_2+\tau+1]\\
&\hspace{0.5cm} \leq \left(1-h(\underline{\psi}+(n^{\infty}-1)\kappa)\right)
\left(\bar{v}^k[t_1]-\frac{\underline{\psi}(1-hn^{\infty} \kappa)^\tau \alpha_d}
{1+\underline{\psi}+(n^{\infty}-1)\kappa}
\frac{\D^k[t_1]}{2}\right)\\
&\hspace{0.5cm} +h(\underline{\psi}+(n^{\infty}-1)\kappa)\bar{v}^k[t_1]
-h\underline{\psi}(1-h n^{\infty} \kappa)^\tau \alpha_d \frac{\D^k[t_1]}{2}.\\
&\hspace{0.5cm} =\bar{v}^k[t_1]-\frac{h\underline{\psi}(1-h n^{\infty} \kappa)^\tau \alpha_d}{1+\underline{\psi}+(n^{\infty}-1)\kappa}
\frac{\D^k[t_1]}{2}.
\end{aligned}
\end{align*}
We combine all estimates in Case B.1 and Case B.2 to get the desired estimate:
\[ v^k_{j} [t_2+\tau+1] \leq \bar{v}^k[t_1]-\frac{h\underline{\psi}\alpha_d(1-hn^{\infty} \kappa)^{2\tau+1}} {1+n^{\infty} \kappa}
\frac{\D^k[t_1]}{2}.
\]
\end{proof}
Now, we are ready to provide the proof of Theorem \ref{T3.2}.

\subsection{Proof of Theorem~\ref{T3.2} } \label{sec:5.2} By Lemma~\ref{L5.1}, we can see that $\D[t]$ is non-increasing with $t$. Next, the proof is divided into two steps. \newline

\noindent $\bullet$~Step A (Exponential decay of relative velocities):  Suppose that network topology ${\mathcal G},$ coupling strength $\kappa$, and time-delay $\tau$ satisfy the relations:
\[ \gamma_g > 0, \quad  n^{\infty} = \max_{1\leq i \leq N} |{\mathcal N}_i| > 0, \quad \tau > 0, \quad \kappa > 0, \quad 0<h \kappa<\frac{1}{n^{\infty}}. \]
For $\rho > 0$ and initial spatial condition $X(0)$, we define $\delta$ as follows.
\begin{equation*}
\delta :=1 -\frac{h^{\gamma_g} (\psi(X(0)+\rho))^{\gamma_g}(1-h n^{\infty} \kappa)^{\gamma_g (3\tau+1)}}{2(1+n^{\infty} \kappa)^{\gamma_g}} >0.
\end{equation*}
For given initial data set $\{ (\bx^0[t], \bv^0[t]):~\tau \leq t \leq 0 \}$,  we show that if there exists a positive constant $\rho>0$ such that
\begin{align}
\begin{aligned} \label{F-4}
\D[0] &\leq \frac{1}{\sqrt{d}h\gamma_g (2\tau+1)} \times \frac{(1-h n^{\infty} \kappa)^{\gamma_g(3\tau+1)}h^{\gamma_g}  (\psi(X[0]+\rho))^{\gamma_g}}{2(1+n^{\infty} \kappa)^{\gamma_g} }  \rho =  \frac{( 1-\delta) \rho}{\sqrt{d}h\gamma_g (2\tau+1)},
\end{aligned}
\end{align}
then we have exponential decay of relative velocities. To this end, it is sufficient to verify
\begin{equation}\label{F-5}
\D[n\gamma_g(2\tau+1)]\leq\delta^n\D[0].
\end{equation}
Next, we verify the relation \eqref{F-5} using mathematical induction on $n$. \newline

\noindent $\diamond$~Step A (Initial step):  Note that  the relation \eqref{F-5} is trivially true.  \newline

\noindent $\diamond$~Step B (induction step): Suppose that the relation \eqref{F-5} holds for all $0\leq n\leq T \in \NN \cup \{0 \}$:
\begin{equation}\label{F-6}
\D[n\gamma_g(2\tau+1)]\leq\delta^n\D[0], \quad 0 \leq n \leq T,
\end{equation}
For $T\gamma_g(2\tau+1)\leq t\leq (T+1)\gamma_g(2\tau+1)$ and $j\in\N_i$, it follows from system~\eqref{F-1} that we have
\begin{align*}
\begin{aligned}
&\|\bx_i[t]-\bx_j[t-\tau_{ij}[t]]\|\\
&\quad\leq\|\bx_i[t-1]-\bx_j[t-\tau_{ij}[t]-1]\|
+h\|\bv_i[t-1]-\bv_j[t-\tau_{ij}[t]-1]\|\\
&\quad\leq\|\bx_i[0]-\bx_j[-\tau_{ij}[t]]\|
+h\sum_{s=0}^{t-1}\|\bv_i[s]-\bv_j[s-\tau_{ij}[t]]\|.\\
\end{aligned}
\end{align*}
Recall that we have
\[ \|\bv_i[s]-\bv_j[s-\tau_{ij}[t]]\|\leq\sqrt{d}\D[s]. \]
Then, we use the same argument as in \eqref{D-17} to find
\begin{align}
\begin{aligned} \label{F-7}
&\|\bx_i[t]-\bx_j[t-\tau_{ij}[t]]\| \\
& \hspace{0.5cm} \leq\|\bx_i[0]-\bx_j[-\tau_{ij}[t]]\|+\sqrt{d}h\sum_{s=0}^{t-1}\D[s] \leq X[0]
+\sqrt{d}h\sum_{n=0}^T\sum_{s=n\gamma_g(2\tau+1)}^{(n+1)\gamma_g (2\tau+1)-1}\D[s]\\
&  \hspace{0.5cm}  \leq X[0] +\sqrt{d}h\gamma_g (2\tau+1)\sum_{n=0}^T\delta^n\D[0] \leq X[0]
+\frac{\sqrt{d}h\gamma_g(2\tau+1)}{1-\delta}\D(0)\\
& \hspace{0.5cm} \leq X[0]+\rho,
\end{aligned}
\end{align}
where the last inequality is from assumption~\eqref{F-4}. \newline

Then, we use \eqref{F-7} and monotonicity of $\psi$ to obtain
\[\label{F-8}
\psi(\|\bx_j[t -\tau_{ij}[t]] - \bx_i[t] \|) \geq \psi(X[0] + \rho).
\]
Since $\G$ has a spanning tree, we take a root $r$ of $\G$ such that $\gamma_g=\max_{j\in\V}\dist(r,j)$. For a fixed $k \in \{1,\cdots, d\}$, without loss of generality, we assume that
\[
v^k_{r}[T\gamma_g(2\tau+1)] \leq \bar{v}^k[T\gamma_g(2\tau+1)]-\frac{\D^k[T\gamma_g(2\tau+1)]}{2}.
\]
Similar to the proof of Theorem \ref{T3.1},  the other case can be analyzed by considering the vector $-v^k[T\gamma_g(2\tau+1)]$. We apply Lemma \ref{L5.2} to have that for all $t\geq T\gamma_g(2\tau+1)$,
\begin{equation}\label{F-9}
v^k_{r}[t]\leq\bar{v}^k[T\gamma_g(2\tau+1)]-
(1-h n^{\infty} \kappa)^{t-T\gamma_g(2\tau+1)}\frac{\D^k[T\gamma_g (2\tau+1)]}{2}.
\end{equation}
For any $l \neq r$, we can find a path from $r$ to $i$ of length $p\leq \gamma_g$:
\[  r=i_0~\rightarrow~i_1~\rightarrow~\ldots~\rightarrow i_p=l. \]
For $i_1$, we apply Proposition \ref{P5.1} with
\[ i=r, \quad j=i_1, \quad t_1=t_2=T\gamma_g (2\tau+1), \quad \alpha_d =1 \quad \mbox{and} \quad \underline{\psi}=\psi(X[0] + \rho) \]
to get
\[ v^k_{i_1}[T\gamma_g (2\tau+1)+\tau+1]\leq\bar{v}^k [T\gamma_g (2\tau+1)]-\frac{h \psi(X[0] + \rho) (1-h n^{\infty} \kappa)^{2\tau+1}}  {1+ n^{\infty} \kappa}\frac{\D^k[T\gamma_g (2\tau+1)]}{2}.
\]
This and Lemma \ref{L5.2} imply
\[ v^k_{i_1}[T\gamma_g (2\tau+1)+2\tau+1]\leq\bar{v}^k[T\gamma_g(2\tau+1)]
-\frac{h\psi(X[0] + \rho) (1-h n^{\infty} \kappa)^{3\tau+1}} {1+n^{\infty} \kappa}\frac{\D^k[T\gamma_g(2\tau+1)]}{2}.
\]
 We continue to apply Proposition \ref{P5.1} with
 \begin{align*}
 \begin{aligned}
&  i=i_1, \quad j=i_2, \quad t_1=T\gamma_g(2\tau+1), \quad t_2=T\gamma_g(2\tau+1)+2\tau+1,  \\
& \alpha_d= \frac{h\psi(X[0] + \rho) (1-hn^{\infty} \kappa)^{3\tau+1}}{1+ n^{\infty} \kappa} \quad \mbox{and} \quad \underline{\psi} :=\psi(X[0] + \rho)
\end{aligned}
\end{align*}
to get
\begin{align*}
\begin{aligned}
&v^k_{i_2}[T\gamma_g (2\tau+1)+2\tau+1+\tau+1]\\
&\hspace{2cm} \leq\bar{v}^k [T\gamma_g (2\tau+1)]
-\frac{h^2 (\psi(X[0] + \rho))^2(1-hn^{\infty} \kappa)^{5\tau+2}}
 {(1+ n^{\infty} \kappa)^2}\frac{\D^k [T\gamma_g (2\tau+1)]}{2}.
\end{aligned}
\end{align*}
We proceed the above arguments until $i_p=l$ to derive
\begin{align*}
\begin{aligned}
&v^k_{l}[T\gamma_g (2\tau+1)+(p-1)(2\tau+1)+\tau+1]\\
&\hspace{1cm} \leq\bar{v}^k[T\gamma_g (2\tau+1)]
-\frac{h^p (\psi(X[0] + \rho))^p(1-h n^{\infty} \kappa)^{(3p-1)\tau+p}}
 {(1+ n^{\infty} \kappa)^p}\frac{\D^k[T\gamma_g (2\tau+1)]}{2}.
\end{aligned}
\end{align*}
This and Lemma \ref{L5.2} imply
\begin{align*}
\begin{aligned}
&v^k_{l}[T\gamma_g (2\tau+1)+(\gamma_g-1)(2\tau+1)+\tau+1]\\
&\hspace{1cm} \leq\bar{v}^k[T\gamma_g (2\tau+1)]
-\frac{h^p (\psi(X[0] + \rho))^p (1-h n^{\infty} \kappa)^{\gamma_g (2\tau+1)+(p-1)\tau}}
 {(1+ n^{\infty} \kappa)^p}\frac{\D^k[T\gamma_g (2\tau+1)]}{2}\\
&\hspace{1cm} \leq\bar{v}^k[T\gamma_g (2\tau+1)]
-\frac{h^{\gamma_g} (\psi(X[0] + \rho))^{\gamma_g} (1-h n^{\infty} \kappa)^{(3\gamma_g-1)\tau+\gamma_g}}
 {(1+n^{\infty} \kappa)^{\gamma_g}}\frac{\D^k[T\gamma_g (2\tau+1)]}{2}.
\end{aligned}
\end{align*}
We recall \eqref{F-9}, and see that the above inequality holds for all $1\leq i\leq k$. We once again apply Lemma \ref{L5.2} to conclude that, for all $t\geq T\gamma_g(2\tau+1)+(\gamma_g-1)(2\tau+1)+\tau+1$ and $i\leq k$, we have
\[
v^k_{i}[t]\leq\bar{v}^k[T\gamma_g(2\tau+1)]
-\frac{h^{\gamma_g}   (\psi(X[0] + \rho))^{\gamma_g} (1-h n^{\infty} \kappa)^{t-T\gamma_g(2\tau+1)+\gamma_g \tau}}
 {(1+ n^{\infty} \kappa)^{\gamma_g}}\frac{\D^k[T\gamma_g(2\tau+1)]}{2}.
\]
Therefore, we have
\begin{align*}
\begin{aligned}
&\bar{v}^k[(T+1)\gamma_g (2\tau+1)]\\
&\hspace{1cm} =\max_{\substack{1\leq i\leq k\\(T+1)\gamma_g (2\tau+1)-\tau\leq t\leq(T+1)\gamma_g (2\tau+1)}}v^k_{i}[t]\\
&\hspace{1cm} \leq\bar{v}^k [T\gamma_g (2\tau+1)]
-\frac{h^{\gamma_g} (\psi(X[0] + \rho))^{\gamma_g} (1-h n^{\infty} \kappa)^{\gamma_g (3\tau+1)}}
 {(1+ n^{\infty} \kappa)^{\gamma_g}}\frac{\D^k[T\gamma_g(2\tau+1)]}{2}
\end{aligned}
\end{align*}
This and induction hypothesis \eqref{F-6} imply
\begin{align*}
\begin{aligned}
&\D^k[(T+1)\gamma_g (2\tau+1)]\\
&\hspace{1cm} =\bar{v}^k[(T+1)\gamma_g(2\tau+1)]-\underline{v}^k[(T+1)\gamma_g(2\tau+1)]\\
&\hspace{1cm} \leq\bar{v}^k[T\gamma_g (2\tau+1)]
-\frac{h^{\gamma_g} (\psi(X[0] + \rho))^{\gamma_g} (1-h n^{\infty} \kappa)^{\gamma_g(3\tau+1)}}
 {(1+ n^{\infty} \kappa)^{\gamma_g}}\frac{\D^k[T\gamma_g(2\tau+1)]}{2} \\
& \hspace{1cm} -\underline{v}^k[T\gamma_g(2\tau+1)]\\
&\hspace{1cm} =\delta\D^k[T\gamma_g (2\tau+1)] \leq\delta^{T+1}\D[0].
\end{aligned}
\end{align*}
Thuse, we have
\[
 |v^k_i[s] - v^k_j[s]|  \leq\D[s]  \leq \delta^{[\frac{s}{\gamma_g (2\tau +1)}]} \D[0],
\]
which implies the velocity flocking:
\[
\lim_{t \to \infty} ||\bv_i[t] - \bv_j[t]|| = 0.
\]

\noindent $\bullet$~Step B (Uniform boundedness of relative positions): By the same arguments in Step B in Section \ref{sec:4.2}, we have
\[ \sup_{-\tau \leq t < \infty} ||\bx_i[t] - \bx_j[t] \|  < \infty.   \]
We combine the estimates in Step A and Step B to obtain the desired mono-cluster flocking estimate.

\section{Numerical Simulations}\label{sec:6}
In this section, we present several numerical simulations to illustrate our theoretical results obtained in previous sections, and  provide further insight into the effects of communication time-delays on the dynamic behavior of the C-S model for initial configurations violating our sufficient condition \eqref{C-1-1} for the continuous C-S model. For numerical simulation, we used the fourth-order Runge-Kutta scheme.

\subsection{Simulation setup} \label{sec:6.1} Consider a flock of four C-S particles moving in $\bbr^2$ by the system~\eqref{C-1} on a communication network (Figure 1):
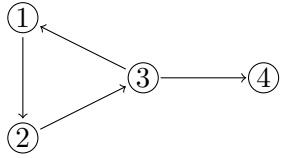
\begin{figure}[H]
\begin{center}
\begin{tikzpicture}[scale=4]
\node (a) at (0,0.2)  {1};
\node (b) at (0,-0.2)  {2};
\node (c) at (0.4,0)  {3};
\node (d) at (0.8,0)  {4};
\draw (a) circle[radius=0.05];
\draw (b) circle[radius=0.05];
\draw (c) circle[radius=0.05];
\draw (d) circle[radius=0.05];
\draw[->]  (a) -- (b);
\draw[->]   (b) -- (c);
\draw[->]   (c) -- (a);
\draw[->]   (c) -- (d);
\end{tikzpicture}
\caption{Digraph connection topology $\C$.}
\end{center}
\vspace{-0.5cm}
\end{figure}
From this network, it is easy to see that the neighbor sets, connection matrix $\C = (\chi_{ij})$ and depth $\gamma_g$ are given as follows.
\begin{align*}
\begin{aligned}
& \C := \left( \begin{array}{cccc}
0 & 0 &1 &0 \\
1 & 0 & 0 &0 \\
0 & 1 & 0 & 0 \\
0 & 0 & 1 & 0
\end{array} \right), \quad {\mathcal N}_1 = \{3 \}, \quad {\mathcal N}_2 := \{1 \}, \\
& {\mathcal N}_3 := \{2 \}, \quad {\mathcal N}_4 := \{3 \}, \quad n^{\infty} =1, \quad \gamma_g = 2.
\end{aligned}
\end{align*}
For the communication weight and time-delay, we take the C-S weight:
\[ \psi(r) = \frac{1}{(1 + r^2)^{\beta}}, \quad \beta \geq 0, \quad \kappa = 1, \quad \tau_{ij} = 1, \quad 1 \leq i, j \leq N, \quad i\neq j. \]
As can be seen in Corollary \ref{C3.1}, for the long-ranged case with $2\beta \gamma_g  < 1$, we have a mono-cluster flocking for any initial configurations. Thus, we only focus on the threshold and short-ranged cases:
\[ 2\beta \gamma_g = 4 \beta =  1, \qquad 2 \beta \gamma_g  = 4 \beta > 1. \]
in the sequel. We set $\{(\bx_i,\bv_i):1\leq i\leq 4\}$ to be the solution of the system \eqref{C-1} over the digraph given in Figure 1, and we set $\{(\tilde\bx_i,\tilde\bv_i):1\leq i\leq 4\}$ to be the solution over the all-to-all network.

\subsection{Threshold case with $\beta = \frac{1}{4}$}
 Our sufficient condition displayed in Corollary \ref{C3.1} says that if initial configuration satisfies
\begin{equation} \label{G-1}
  \D(0) <  \frac{e^{-10}}{48 \sqrt{2}}, \quad X(0) < \infty
\end{equation}
then, we have an exponential mono-cluster synchronization. We performed our first numerical experiment with the initial data satisfying the condition \eqref{G-1}. Note that the maximal time-delay is unity, i.e., $\tau = 1$.  The fixed initial data for the two flows $\{(\bx_i,\bv_i):1\leq i\leq 4\}$ and $\{(\tilde\bx_i,\tilde\bv_i):1\leq i\leq 4\}$, for all $-1 \leq t\leq0$,  are given in the table below.
\begin{center}
\begin{tabular}{||c|c||c|c||c|c||c|c||}
  \hline
$\bx_1(t)$ & $(1, 0)$ & $\bv_1(t)$ & $\frac{e^{-10}}{672 \sqrt{2}}(1, -2)$
&  $\bx_2(t)$ & $(0, 1)$ & $\bv_2(t)$ & $\frac{e^{-10}}{672 \sqrt{2}}(3, -4)$\\
 $\bx_3(t)$ & $(-1, 0)$ & $\bv_3(t)$ & $\frac{e^{-10}}{672 \sqrt{2}}(5, 6)$
 & $\bx_4(t)$ & $(0, -1)$ & $\bv_4(t)$ & $\frac{e^{-10}}{672 \sqrt{2}}(-7, -8)$\\
  \hline
\end{tabular}
 \end{center}

In Figure 2, we compare the dynamic behavior of the first velocity components over the given digraph and all-to-all network, respectively.
\begin{figure}[H]
\centering
\includegraphics[width=1.0\textwidth]{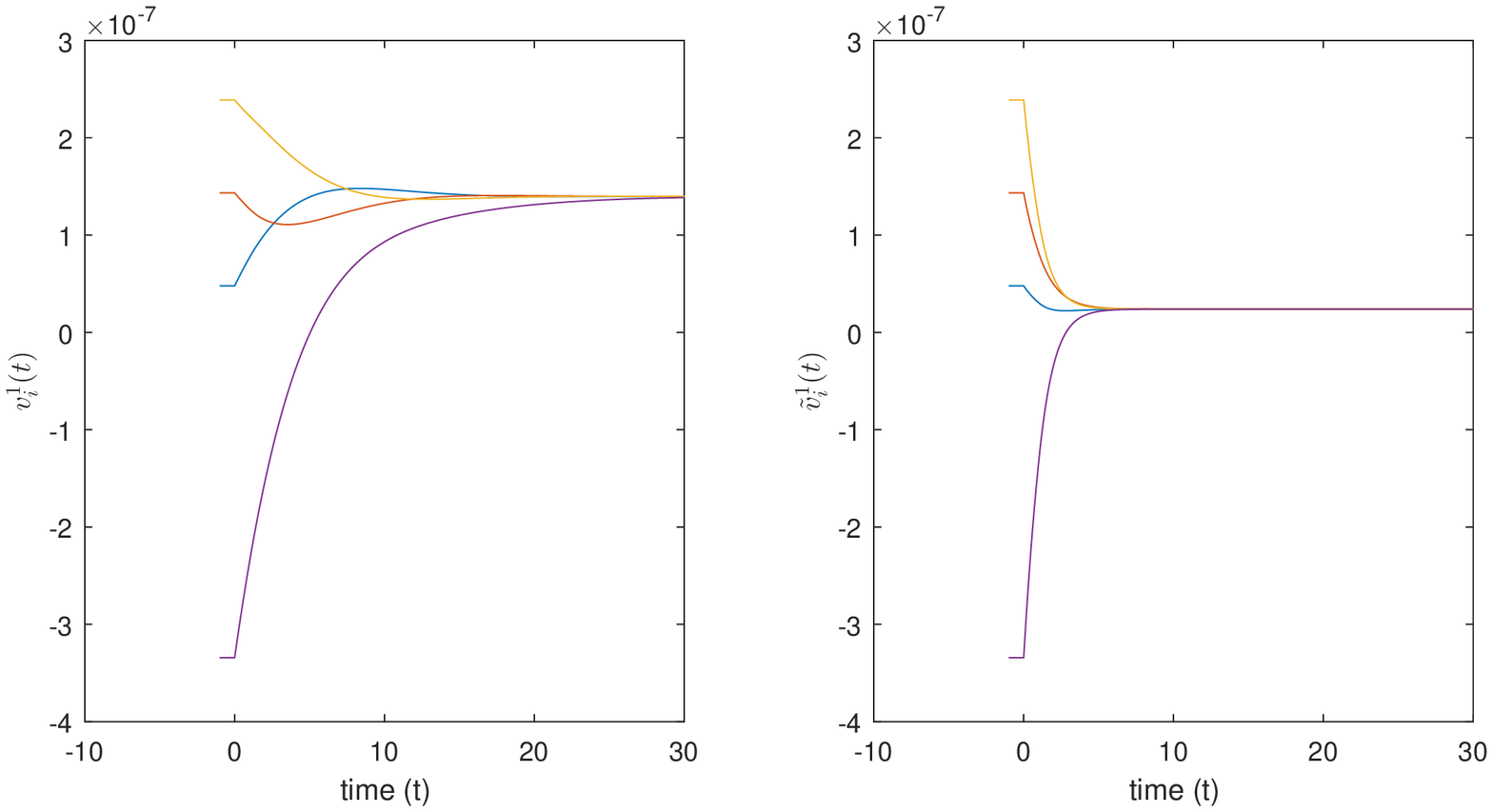}
\caption{The convergence trajectories of the first component velocities satisfying the condition \eqref{G-1}. Left: Digraph $\C$ and right: all-to-all graph}
\label{figu2}
\end{figure}

As expected, two flows both exhibit mono-cluster flocking, and the all-to-all network case flocks faster than the digraph case.
However, note that the condition \eqref{G-1} is not necessary for the flocking to occur. In our second numerical experiment, we set the initial data for the two flows $\{(\bx_i,\bv_i):1\leq i\leq 4\}$ and $\{(\tilde\bx_i,\tilde\bv_i):1\leq i\leq 4\}$, for all $-\tau\leq t\leq0$, as follows.
\begin{center}
\begin{tabular}{||c|c||c|c||c|c||c|c||}
  \hline
$\bx_1(t)$ & $(1, 0)$ & $\bv_1(t)$ & $(1, -2)$
&  $\bx_2(t)$ & $(0, 1)$ & $\bv_2(t)$ & $(3, -4)$\\
 $\bx_3(t)$ & $(-1, 0)$ & $\bv_3(t)$ & $(5, 6)$
 & $\bx_4(t)$ & $(0, -1)$ & $\bv_4(t)$ & $(-7, -8)$\\
  \hline
\end{tabular}
 \end{center}

In Figure 3, we again compare the dynamic behavior of the first velocity components over the given digraph and all-to-all network, respectively.
\begin{figure}[H]
\centering
\includegraphics[width=1.0\textwidth]{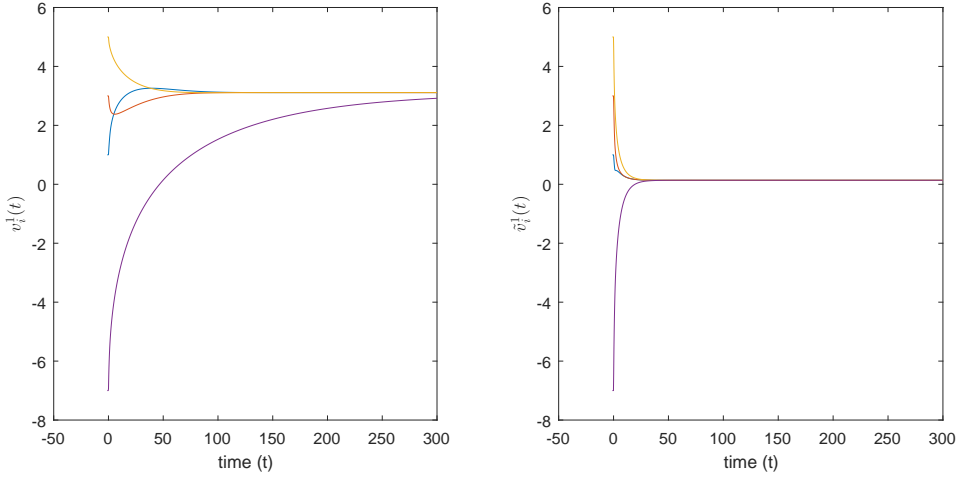}
\caption{The convergence trajectories of the first component velocities not satisfying the condition \eqref{G-1}. Left: Digraph ${\mathcal C}$ and right: all-to-all graph}
\label{figu3}
\end{figure}
Based on numerical simulations, as in the C-S model with all-to-all network, we conjecture that for the digraph, the critical threshold will be $\beta = \frac{1}{2\gamma_g}$ from the mono-cluster flocking view point, i.e., if $2\gamma_g\beta > 1$, we might have multi-cluster flocking depending on initial configurations, whereas for $2\gamma_g\beta \leq 1$, we have a mono-cluster flocking for any initial configurations.

\subsection{Threshold case with $\beta=\frac{17}{32} > \frac{1}{4}$}  Next, we describe our simulation setup as follows. The initial configuration $\{(\bx_i,\bv_i):1\leq i\leq 4\}$ and $\{(\tilde\bx_i,\tilde\bv_i):1\leq i\leq 4\}$ are given as in the following table:
\begin{center}
\begin{tabular}{||c|c||c|c||c|c||c|c||}
  \hline
$\bx_1(t)$ & $(1, 0)$ & $\bv_1(t)$ & $\frac{e^{-10}}{7056 \sqrt{2}}(1, -2)$
&  $\bx_2(t)$ & $(0, 1)$ & $\bv_2(t)$ & $\frac{e^{-10}}{7056 \sqrt{2}}(3, -4)$\\
 $\bx_3(t)$ & $(-1, 0)$ & $\bv_3(t)$ & $\frac{e^{-10}}{7056 \sqrt{2}}(5, 6)$
 & $\bx_4(t)$ & $(0, -1)$ & $\bv_4(t)$ & $\frac{e^{-10}}{7056 \sqrt{2}}(-7, -8)$\\
  \hline
\end{tabular}
 \end{center}
Then $X(0)=\rho_+=2$, so by some calculation we deduce that the above initial data satisfy the sufficient condition for mono-cluster flocking to occur displayed in Corollary \ref{C3.1}. In Figure 4, we can see that two flows exhibit mono-cluster flocking as expected.

\begin{figure}[H]
\centering
\includegraphics[width=1.0\textwidth]{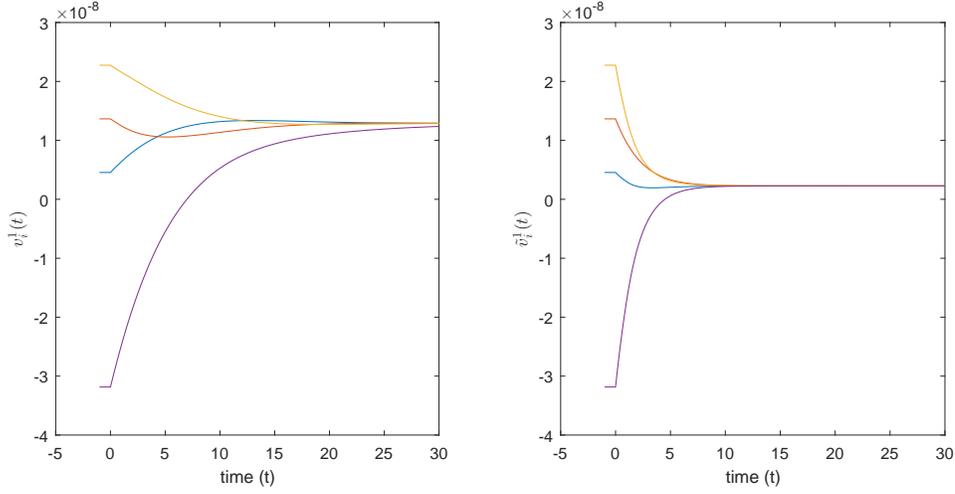}
\caption{The convergence trajectories of the first component velocities satisfying the condition in Corollary \ref{C3.1}.}
\label{figu4}
\end{figure}
However, if initial data for two flows do not satisfy the condition given in Corollary \ref{C3.1}, then the flocking might not occur. Suppose that initial data for two flows $\{(\bx_i,\bv_i):1\leq i\leq 4\}$ and $\{(\tilde\bx_i,\tilde\bv_i):1\leq i\leq 4\}$, for all $-\tau\leq t\leq0$ are given as in the following table:
\begin{center}
\begin{tabular}{||c|c||c|c||c|c||c|c||}
  \hline
$\bx_1(t)$ & $(1, 0)$ & $\bv_1(t)$ & $(1, -2)$
&  $\bx_2(t)$ & $(0, 1)$ & $\bv_2(t)$ & $(3, -4)$\\
 $\bx_3(t)$ & $(-1, 0)$ & $\bv_3(t)$ & $(5, 6)$
 & $\bx_4(t)$ & $(0, -1)$ & $\bv_4(t)$ & $(-7, -8)$\\
  \hline
\end{tabular}
 \end{center}
In Figure 5, we can see that the two flows do not exhibit mono-cluster flocking.

\begin{figure}[H]
\centering
\includegraphics[width=1.0\textwidth]{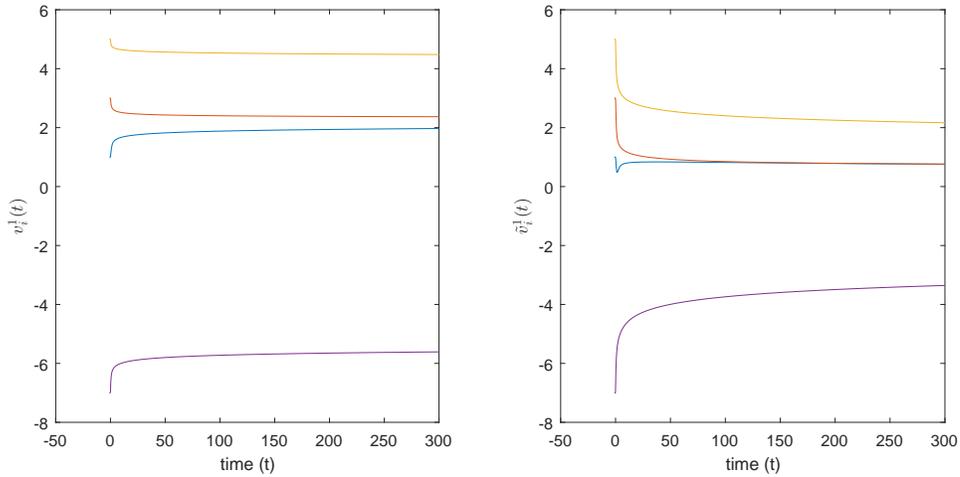}
\caption{The trajectories of the first component velocities not satisfying the condition in Corollary \ref{C3.1}.}
\label{figu5}
\end{figure}

\section{Conclusion}\label{sec:7}
In this paper, we studied the emergent dynamics of the continuous and discrete Cucker-Smale models on digraph connection topology under the effect of time-varying time delays. For a mono-cluster flocking behavior of the ensemble, we assume that our digraph contains a spanning tree, i.e., there is a root whose information flows into any particle along a finite length path. We explicitly provide sufficient frameworks for the  C-S models, which lead to the mono-cluster flocking in terms of system parameters and initial data. In particular, for the explicit C-S communication weights, we provided explicit conditions on the size of initial data. Our analytical results show the robustness of asymptotic mono-cluster flocking  with respect to time-delays. There are other interesting remaining questions. For example, we have not addressed the emergence of multi-cluster flockings which can occur, when our sufficient frameworks are not valid. This can be seen in numerical simulations in Section \ref{sec:6}. These issues will be discussed in future work.


\begin{thebibliography}{99}

\bibitem{A-C-H-L} S. Ahn, H. Choi, S.-Y. Ha and H. Lee: \textit{On the collision avoiding initial-configurations to the Cucker-Smale type flocking models.}
 Comm. Math. Sci. {\bf 10}, 625-643 (2012).

\bibitem{A-H} S. Ahn and S.-Y. Ha: \textit{Stochastic flocking dynamics of the Cucker-Smale model with
multiplicative white noises.} J. Math. Phys. {\bf 51},  103301 (2010).

\bibitem{AFH09}
M. Aouchiche, O. Favaron and P. Hansen: \textit{Variable neighborhood search for extremal graphs. 22. Extending
bounds for independence to upper irredundance.} Discret Appl. Math. {\bf 157} 3497-3510 (2009).

\bibitem{BCGLOPVZ08}
M. Ballerini, N. Cabibbo, R. Candelier, A. Cavagna, E. Cisbani, I. Giardina,
V. Lecomte, A. Orlandi, G. Parisi, A. Procaccini, M. Viale, V. Zdravkovic: \textit{Interaction ruling animal collective behavior depends on topological rather than metric distance: evidence from a field study.} Proc. Natl. Acad. Sci. USA {\bf 105}, 1232-1237 (2008).

\bibitem{BLH01} R. W. Beard, J. Lawton, and F. Y. Hadaegh: \textit{A coordination architecture for spacecraft formation control.}
IEEE Trans. Control Syst. Technol. {\bf 9}, 777-790 (2001).

\bibitem{B-C-C} F. Bolley, J. A. Canizo and J. A. Carrillo: \textit{Stochastic mean-field
limit: non-Lipschitz forces and swarming.} Math. Mod. Meth. Appl. Sci. {\bf 21}, 2179-2210 (2011).

\bibitem{C-C-R} J. A. Canizo, J. A. Carrillo and J. Rosado: \textit{A well-posedness theory
in measures for some kinetic models of collective motion.} Math. Mod. Meth. Appl. Sci. {\bf 21}, 515-539 (2011).


\bibitem{C-F-R-T} J. A. Carrillo, M. Fornasier, J. Rosado and G. Toscani: \textit{Asymptotic
flocking dynamics for the kinetic Cucker-Smale model.} SIAM J. Math. Anal. {\bf 42}, 218-236 (2010).


\bibitem{C-H-H-J-K1} J. Cho, S.-Y. Ha, F. Huang, C. Jin and D. Ko: \textit{Emergence of bi-cluster flocking for the Cucker-Smale model.} Math. Models Methods Appl. Sci. {\bf 26},1191-1218 (2016).


\bibitem{CH17} Y.-P. Choi and J. Haskovec: \textit{Cucker-Smale model with normalized communication weights and time
delay.} Kinetic Relat. Models {\bf 10}, 1011-1033 (2017).

\bibitem{C-L} Y.-P. Choi and Z. Li: \textit{Emergent behavior of Cucker-Smale flocking particles with time-lags.} Appl. Math. Lett. {\bf86}, 49-56 (2018).

\bibitem{CMKB04} J. Cort\'{e}s, S. Martinez, T. Karatas, and F. Bullo: \textit{Coverage control for mobile sensing networks.}
IEEE Trans. Robot. Autom. {\bf 20}, 243-255 (2004).

\bibitem{CKFS05}
I. D. Couzin, J. Krause, N. R. Franks, and S. Levin: \textit{Effective leadership and decision making in animal groups on the move}. Nature {\bf 433}, 513-516 (2005).

\bibitem{CD10} F. Cucker and J.-G. Dong: \textit{Avoiding collisions in flocks.}  IEEE Trans. Automat. Control {\bf 55}, 1238-1243 (2010).

\bibitem{C-M} F. Cucker and E. Mordecki: \textit{Flocking in noisy environments.} J. Math. Pure Appl. {\bf 89}, 278-296 (2008).

\bibitem{CS07} F. Cucker and S. Smale: \textit{Emergent behavior in flocks.} IEEE Trans. Automat. Control {\bf 52}, 852-862 (2007).

\bibitem{CS071} F. Cucker and S. Smale:  \textit{On the mathematics of emergence.}  Japan. J. Math. {\bf 2}, 197-227 (2007).

%
%

\bibitem{DQ17}
J.-G. Dong and L. Qiu: \textit{Flocking of the Cucker-Smale model on general digraphs.} IEEE Trans. Automat. Control {\bf 62}, 5234-5239 (2017).

\bibitem{D-F-T} R. Duan, M. Fornasier and G. Toscani: \textit{A kinetic flocking model with
diffusion.} Commun. Math. Phys. {\bf 300},  95-145 (2010).

\bibitem{E-H-S} R. Erban, J. Haskovec and Y. Sun: \textit{On Cucker-Smale model with noise and delay.}  SIAM. J. Appl. Math. {\bf 76}, 1535-1557 (2016).

\bibitem{F-H-T} M. Fornasier, J. Haskovec and G. Toscani: \textit{Fluid dynamic description
of flocking via Povzner-Boltzmann equation.} Phys. D {\bf 240}, 21-31 (2011) .


\bibitem{H-K-Z} S.-Y. Ha, D. Ko and Y. Zhang: \textit{Critical coupling strength of the Cucker-Smale model for flocking.} Math. Models Methods Appl. Sci. {\bf 27}, 1051-1087 (2017).

\bibitem{Hale-L} J. K. Hale and S. M. V. Lunel: \textit{Introduction to functional differential equations.} Applied Mathematical Sciences {\bf 99}, Springer. 1993

\bibitem{HLL09}
S.-Y. Ha, K. Lee and D. Levy: \textit{Emergence of time-asymptotic flocking in a stochastic Cucker-Smale system.} Commun. Math. Sci. {\bf 7}, 453-469 (2009).

\bibitem{LW14} Y. Liu and J. Wu: \textit{Flocking and asymptotic velocity of the Cucker-Smale model with processing delay.}
 J. Math. Anal. Appl. {\bf 415}, 53-61 (2014).

\bibitem{H-Liu} S.-Y. Ha and J.-G. Liu:  \textit{A simple proof of Cucker-Smale flocking dynamics and mean field limit.} Commun. Math. Sci. {\bf 7}, 297-325 (2009).



\bibitem{H-T} S.-Y. Ha and E. Tadmor: \textit{From particle to kinetic and hydrodynamic description of
flocking.} Kinetic Relat. Models {\bf 1}, 415-435  (2008).


\bibitem{JLM03} A. Jadbabaie, J. Lin, and A. Morse: \textit{Coordination of groups of mobile autonomous agents using nearest neighbor rules.}
IEEE Trans. Automat. Control {\bf 48}, 988-1001 (2003).

%
%

\bibitem{L-X}  Z. Li and X. Xue: \textit{Cucker-Smale flocking under rooted leadership with fixed and switching topologies.} SIAM J. Appl. Math. {\bf 70}, 3156-3174 (2010).

%
\bibitem{M-T} S. Motsch and E. Tadmor: \textit{A new model for self-organized dynamics and its flocking
behavior.} J. Stat. Phys. {\bf 144}, 923-947 (2011).

\bibitem{OS06} R. Olfati-Saber: \textit{Flocking for multi-agent dynamic systems: algorithms and theory.}  IEEE Trans. Automat. Control {\bf 51}, 401-420 (2006).

\bibitem{PEG09}
L. Perea, P. Elosegui and G. G\'{o}mez: \textit{Extension of the Cucker-Smale control law to space fight formations.} J. Guidance Contr. Dyn. {\bf 32}, 526-536 (2009).


\bibitem{P-K-H} J. Park, H. Kim and S.-Y. Ha: \textit{Cucker-Smale flocking with inter-particle bonding forces.}
IEEE Tran. Automat. Control {\bf 55}, 2617-2623 (2010).

\bibitem{PT17}
C. Pignotti and E. Trelat: \textit{Convergence to consensus of the general finite-dimensional Cucker-Smale model with time-varying delays.} arXiv:1707.05020 (2017).

\bibitem{PV17} C. Pignotti and I. R. Vallejo: \textit{Flocking estimates for the Cucker-Smale model with time lag and hierarchical leadership.} J. Math. Anal. Appl. {\bf464}, 1313-1332  (2018).

\bibitem{P-S} D. Poyato and J. Soler: \textit{Euler-type equations and commutators in singular and hyperbolic limits of kinetic Cucker-Smale models.} Math. Mod. Meth. Appl. Sci. {\bf 6},1089-1152 (2017).

\bibitem{Rey87} C. W. Reynolds: \textit{Flocks, herds, and schools: A distributed behavioral model.} Comput. Graph {\bf 21}, 25-34 (1987).

\bibitem{She} J. Shen: \textit{Cucker-Smale flocking under hierarchical leadership.}
SIAM J. Appl. Math. {\bf 68}, 694-719 (2007).

\bibitem{TJJ07} H. G. Tanner, A. Jadbabaie, and G. J. Pappas: \textit{Flocking in fixed and switching networks.} IEEE Trans. Automat. Control {\bf 52}, 863-868 (2007).

%

\bibitem{VVSTSNV14}
 G.  V\'{a}s\'{a}rhelyi, C. Vir\'{a}gh, G. Somorjai, N.
Tarcai, T. Sz\"{o}r\'{e}nyi, T. Nepusz, and T. Vicsek: \textit{Outdoor flocking and formation flight with autonomous
aerial robots.} Proc. IEEE/RSJ Int. Conf. Intell. Robots
Syst. 3866-3873 (2014).

\bibitem{VKJS95} T. Vicsek, A. Czir\'ok, E. Ben-Jacob, I. Cohen and O. Schochet: \textit{Novel type of phase transition in a system of self-driven particles.} Phys. Rev. Lett. {\bf 75}, 1226-1229 (1995).

\bibitem{VZ12} T. Vicsek and A. Zefeiris: \textit{Collective motion.}  Phys. Rep. {\bf 517}, 71-140 (2012).

\bibitem{VVTSSNV14}
C. Vir\'{a}gh, G. V\'{a}s\'{a}rhelyi, N. Tarcai, T. Sz\"{o}r\'{e}nyi, G. Somorjai, T. Nepusz, and T. Vicsek: \textit{Flocking algorithm for autonomous flying robots.}  Bioinspir. Biomim. {\bf 9}, 025012 (2014).


\end{thebibliography}
\end{document}